\documentclass[11pt]{article}

\usepackage{amsmath,amsthm}

\usepackage{amssymb,latexsym}

\usepackage{enumerate}

\newtheorem{tw}{Theorem}[subsection]
\newtheorem{lm}[tw]{Lemma}
\newtheorem{wn}[tw]{Corollary}
\newtheorem{stw}[tw]{Proposition}
\newenvironment{dow}{\it Proof.\rm}{\hfill $\Box$}

\theoremstyle{definition}
\newtheorem*{df}{Definition}
\newtheorem{uw}[tw]{Remark}
\newtheorem{prz}[tw]{Example}

\newcommand{\BN}{{\mathbb N}}
\newcommand{\BR}{{\mathbb R}}

\newcommand{\VV}{{\mathcal V}}
\newcommand{\TT}{{\mathcal T}}

\newcommand{\FF}{{\mathcal{F}}}

\newcommand{\BB}{{\mathcal{B}}}

\newcommand{\MM}{{\mathcal{M}}}

\newcommand{\BRD}{{\mathbb{R}^{d}}}

\newcommand{\essinf}{\mathop{\mathrm{ess\,inf}}}
\newcommand{\esssup}{\mathop{\mathrm{ess\,sup}}}

\newcommand{\nsubsection}{\setcounter{equation}{0}\subsection}

\begin{document}
\title {Reflected BSDEs on Filtered Probability Spaces}
\author {Tomasz Klimsiak \smallskip\\
{\small Faculty of Mathematics and Computer Science,
Nicolaus Copernicus University} \\
{\small  Chopina 12/18, 87--100 Toru\'n, Poland}\\
{\small e-mail: tomas@mat.uni.torun.pl}}
\date{}
\maketitle
\begin{abstract}
We study  the problem of existence and uniqueness of solutions of
backward stochastic differential equations  with two reflecting
irregular barriers, $L^p$ data and generators satisfying weak
integrability conditions. We deal with equations on general
filtered probability spaces. In case the generator does not depend
on the $z$ variable, we first consider the case  $p=1$ and we only
assume that the underlying filtration satisfies the usual
conditions of right-continuity and completeness. Additional
integrability properties of solutions are established if
$p\in(1,2]$ and the filtration is quasi-continuous. In case the
generator depends on $z$, we assume that $p=2$, the filtration
satisfies the usual conditions and additionally that it is
separable. Our results apply for instance to Markov-type reflected
backward equations driven by general Hunt processes.
\end{abstract}

\footnotetext{{\em MSC2010 subject classifications.} Primary
60H10; Secondary 60H20.}


\footnotetext{Research supported by Polish NCN grant no.
2012/07/B/ST1/03508.}

\nsubsection{Introduction}

In the present paper we study  the problem of existence and
uniqueness of solutions of  backward stochastic differential
equations (BSDEs for short) with two reflecting barriers. There is
now an extensive literature on the subject, but unfortunately all
the available results concern equations with underlying filtration
generated by a Wiener process (Brownian filtration) or by a
Poisson random measure and an independent Wiener process
(Brownian-Poisson filtration). In the paper we deal with equations
on general filtered probability spaces. In the case where the
generator of the equation is independent of the $z$ variable we
only assume that the underlying filtration
$\FF=\{\FF_t;t\in[0,T]\}$ satisfies the usual conditions of
right-continuity and completeness. In the general case of
equations with generators depending on $z$ we assume that the
Hilbert space $L^2(\FF_T)$ is separable.  Another dominant feature
of the paper is that we study equations with irregular barriers,
$L^p$ data ($p\in[1,2]$ in case the generator is independent of
$z$ and $p=2$ in the general case) and with generators satisfying
weak regularity and growth assumptions.

In the case of Brownian filtration the theory of reflected BSDEs
with irregular barriers and weak assumptions on the data is quite
well developed. We refer the reader to
\cite{Hamadene,Lepeltier,PengXu} for existence and uniqueness
results for equations with irregular barriers. Reflected BSDEs
with monotone generator satisfying weak growth condition are
studied in \cite{Kl:BSM,Kl:EJP,LepeltierMatoussiXu,RS:SPA},
whereas equations with $L^p$-data and $p\in[1,2]$ in
\cite{HamadenePopier,Kl:BSM,Kl:EJP,RS:SPA}. In the case of the
Brownian-Poisson filtration the only known results concern
reflected BSDEs with c\`adl\`ag barriers, Lipschitz-continuous
generators and $L^2$ data (see
\cite{HamadeneHassani,HamadeneWang}).

Let $(\Omega,\FF=\{\FF_t;t\in [0,T]\},P)$ be a filtered
probability space satisfying the usual conditions. Suppose we are
given an $\FF_T$ measurable random variable $\xi$ (terminal time),
a measurable function $f:\Omega\times
[0,T]\times\mathbb{R}\rightarrow\BR$ (generator) such that
$f(\cdot,y)\in Prog([0,T]\times\Omega)$ and two progressively
measurable processes $L,U$ (barriers) such that $L_t\le U_t$ for
a.e. $t\in [0,T]$. By a solution of the reflected BSDE with data
$\xi,f$ and barriers $U,L$ (RBSDE$(\xi,f,L,U)$) for short) on
$(\Omega,\FF,P)$ we mean a triple $(Y,M,R)$ consisting of an
adapted c\`adl\`ag process $Y$ of Doob's class (D), a local
martingale $M$ with $M_0=0$ and a predictable finite variation
process $R$ with $R_0=0$ such that
\begin{align}
\label{eqi.1} Y_t=\xi+\int_t^T f(r,Y_r)\,dr+\int_t^T dR_r-\int_t^T
dM_r,\quad t\in [0,T],
\end{align}
\begin{equation}
\label{eq1.02} L_t\le Y_t\le U_t\quad \mbox{for a.e. }t\in[0,T]
\end{equation}
and the following minimality condition for $R$ is satisfied: for
every c\`adl\`ag  processes $\hat{L}, \hat{U}$ such that
$L_t\le\hat{L}_t\le Y_t\le\hat{U}_t\le U_t$ for a.e. $t\in [0,T]$
we have
\begin{equation}
\label{eqi.2} \int_0^T (Y_{t-}-\hat{L}_{t-})\,dR^{+}_t =\int_0^T
(\hat{U}_{t-}-Y_{t-})\,dR^{-}_t=0,
\end{equation}
where $R=R^+-R^-$ is the Jordan decomposition  of the measure
$dR$. Condition (\ref{eqi.2}) has been considered in \cite{PengXu}
in the case of reflected BSDEs with Brownian filtration. Note that
the above definition of a solution is similar in spirit to the
definition of a solution of nonreflected BSDEs on general filtered
spaces considered in \cite{KR:JFA,LLQ}. It is well suited for
studying by probabilistic methods partial differential equations
with irregular data (see \cite{KR:JFA,KR:SM}).

In the paper we assume that $f$ is continuous, monotone with
respect to $y$ and satisfies the following mild growth condition
\begin{equation}
\label{eqi.3} E\int_0^T|f(t,0)|\,dt<\infty,\quad
\forall\,{y\in\BR},\,\,\,[0,T]\ni t\mapsto f(t,y)\in L^1(0,T).
\end{equation}
Condition (\ref{eqi.3}) has appeared before in the papers on
nonreflected (see \cite{BDHPS}) and reflected (see
\cite{Kl:BSM,Kl:EJP}) BSDEs with $L^1$ data adapted to the
Brownian filtration. As for the barriers, we merely assume that
they are measurable and satisfy the following Mokobodski
condition: there exists a special semimaringale $X$ with
integrable finite variation part such that
\begin{equation}
\label{eqi.4} L_t\le X_t\le U_t\quad \mbox{for a.e. }t\in [0,T],
\quad E\int_0^T|f(t,X_t)|\,dt<\infty.
\end{equation}
We prove that if $f,L,U$ satisfy conditions (\ref{eqi.3}), (\ref{eqi.4}) and
the data are in $L^1$, i.e. $\xi\in L^1(\FF_T)$ and $\int_0^T|f(\cdot,0)|\,dt\in
L^1(\FF_T)$, then there exists a unique solution $(Y,M,R)$ of
RBSDE$(\xi,f,L,U)$. We also show that under the assumptions
ensuring the existence of a solution of nonreflected BSDE
condition (\ref{eqi.4}) is necessary   for the existence of a
solution of (\ref{eqi.1}) such that $E|R|_T<\infty$. Furthermore,
we show that if the jumps of the barriers are totally inaccessible
and $\FF$ is quasi-left continuous then $R$ is continuous
(reflected BSDEs with such barriers and Poisson-Brownian
filtration are considered in \cite{HamadeneHassani}). Finally, we
show that if the barriers satisfy the standard Mokobodski
condition, i.e. only the first condition in (\ref{eqi.4}) is
satisfied, then the solution still exists but in general $R$ is
not integrable (it may happen that $E|R|^q_T=\infty $ for every
$q>0$, see \cite{Kl:EJP}). In Section 5 we show that under the
additional assumption of quasi-left continuity of the filtration
$\FF$, if the data are $L^p$ integrable for some $p\in (1,2]$,
i.e. $\xi\in L^p(\FF_T)$, $\int_0^T|f(\cdot,0)|\,dt\in L^p(\FF_T)$,
$X\in\mathcal{H}^p$ and $\int_0^T|f(t,X_t)|\,dt\in L^p(\FF_T)$,
then  the solution $(Y,M)$ of (\ref{eqi.1})--(\ref{eqi.2}) belongs
to the space $\mathcal{S}^p\otimes\mathcal{M}^p$.

In the last section of the paper we study BSDEs with generators
possibly depending on the $z$ variable. To deal with such
equations we need some sort of the representation theorem for
square integrable martingales. In the paper we use the
representation by series of stochastic integrals,
 because it  applies to general
filtered spaces. In the context of BSDEs this  type of
representation of martingales has been used in
\cite{CohenElliott,NualartSchoutens}. It should be stressed,
however, that our methods also works for other type of
representations. For instance, one can employ the representation
of \cite{TangLi}, which is commonly used in the case of BSDEs with
the Brownian-Poisson filtration (see Remark \ref{uw5.1}).

To have the representation theorem, we assume additionally that
$L^2(\Omega,\FF_T,P)$ is a separable Hilbert space. It is known
that then there exists an orthogonal sequence $\{M^i\}$ of square
integrable martingales such that each locally square integrable
martingale $N$ admits the representation
\begin{equation}
\label{eqi.5} N_t=N_0+\sum_{i=1}^{\infty}\int_0^t
Z^i_r\,dM^i_r,\quad t\in [0,T]
\end{equation}
for some sequence $\{Z^i\} $ of predictable processes. By a
solution of RBSDE$(\xi,f,L,U)$ we mean a triple $(Y,Z,R)$
consisting of a c\`adl\`ag adapted process $Y\in \mathcal{S}^2$,
predictable process $Z=\{Z^i\}$ satisfying
$P(\sum_{i=1}^{\infty}\int_0^T|Z^i_t|^2\,d\langle
M^i\rangle_t<\infty)=1$ and a predictable finite variation process
$R$ with $R_0=0$, such that (\ref{eq1.02}), (\ref{eqi.2}) hold
true and
\[
Y_t=\xi+\int_t^T f(r,Y_r,Z_r)\,dr+\int_t^T dR_r
-\sum_{i=1}^{\infty}\int_t^T Z^i_r\,dM^i_r,\quad t\in [0,T].
\]
In our main result of the last section we assume that the data are
in $L^2$, i.e. $\xi\in L^2(\FF_T)$ and $\int_0^T|f(\cdot,0,0)|\, dt\in L^2(\FF_T)$,
and that the generator $f$ is monotone with respect to $y$ and
Lipschitz continuous with respect to $z$ (in the appropriate norm
generated by the sequence $\{M^i\}$), satisfies the growth
condition with respect to $y$ similar to (\ref{eqi.3}) and
\[
L_t\le X_t\le U_t\quad \mbox{for a.e. }t\in [0,T], \quad
E\big(\int_0^T|f(t,X_t,0)|\,dt\big)^2<\infty
\]
for some semimartingale $X\in \mathcal{H}^2$. We show that under
these assumptions there exists a unique solution $(Y,Z,R)\in
\mathcal{S}^2\otimes M^2\otimes\mathcal{V}^2$ of
RBSDE$(\xi,f,L,U)$.

We close the presentation of our main results with the following
general remarks.  In the existing literature mainly reflected
BSDEs on spaces equipped with  Brownian or Brownian-Poisson
filtration are considered. One of the reasons is that such a
framework is sufficient for applications of BSDEs to mathematical
finance. In the present paper we improve these  results on RBSDEs
by relaxing assumptions on the generator. Namely, we replace the
Lipschitz continuity of $f$ with respect to $y$ by monotonicity
and as a growth condition we only impose very weak integrability
condition (\ref{eqi.3}). Moreover, if the generator is independent
of $z$, we consider $L^p$ data for $p\in [1,2)$. Reflected BSDEs
on more general filtered probability spaces arise naturally in
applications to partial differential equations involving operators
generated by semi-Dirichlet forms or generalized Dirichlet forms.
The papers \cite{KR:JFA,KR:SM} show that BSDEs provide very
efficient tool for investigating  abstract elliptic equations of the form
\begin{equation}
\label{eq.ii} -Au=f(x,u)+\mu,
\end{equation}
where $\mu$ is a smooth measure and $A$ is a Dirichlet operator.
Similar to (\ref{eq.ii}) parabolic equations are investigated in
\cite{Kl:JFA}.
To study (\ref{eq.ii}) one needs to consider backward equations
with  forward driving process being a general special standard
Markov processes. Our main motivation for studying  BSDEs in an
abstract framework was to cover this class of processes. Also note
that in the whole paper we consider generalized reflected RBSDEs,
i.e. equations of the form (\ref{eqi.1}) perturbed by some finite
variation process $V$. In applications to PDEs we have in mind,
adding $V$ to (\ref{eqi.1}) allows one to study equations with
true measure data ($V$ is then the additive functional of the
forward process in the Revuz correspondence with the measure on
the right-hand side of the equation).

\nsubsection{BSDEs with one reflecting barrier}
\label{sec2}

Let us fix $T>0$ and a stochastic basis $(\Omega, \FF=\{\FF_{t},
t\in[0,T]\},P)$  satisfying the usual conditions of right
continuity and completeness. By $\mathcal{T}$ we denote the set of
all $\FF$ stopping times with values in $[0,T]$. For $s,t\in
[0,T]$ such that $s\le t$ we denote by $\mathcal{T}_{t}$ (resp.
$\mathcal{T}_{s,t}$) the set of all  $\tau\in\mathcal{T}$ such
that $P(\tau\in [t,T])=1$ (resp. $P(\tau\in [s,t])=1$).

By $\MM$ (resp. $\MM_{loc}$) we denote the space of $\FF$
martingales (resp. local $\FF$ martingales). $\MM_0$ (resp.
$\MM^p$ ) is the subspace of $M\in\MM$ consisting of $M$ such that
$M_0=0$ (resp. $E[M]^{p/2}_T<\infty$). By $\VV$ (resp. $\VV^{+}$)
we denote the space of $\FF$ progressively measurable processes of
finite variation (resp. increasing). $\VV_0$ (resp. $\VV^p$) is
the subspace of $\VV$ consisting of $V$ such that $V_0=0$ (resp.
$E|V|^p_T<\infty$). $^{p}\VV$ is the space of predictable
processes in $\VV$. By $\mathcal{S}^{p}$ we denote the space of
$\FF$ progressively measurable processes $Y$ such that
$E\sup_{t\le T}|Y_t|^{p}<\infty$. By $L^p(\FF)$ we denote the
space of $\FF$ progressively measurable processes $X$ such that
$E\int_{0}^T|X_t|^p\,dt<\infty$. $L^p(\FF_T)$ is the space of
$\FF_T$ measurable random variables $X$ such that
$E|X|^p_T<\infty$.

In the rest of the paper $\xi$ is an $\FF_{T}$ measurable random
variable, $L$ is an $\FF$-progressively measurable process,
$V\in\mathcal{V}_0$, $f:\Omega\times[0,T]\times\BR\rightarrow\BR$
is a measurable function such that  $f(\cdot,y)$ is $\FF$
progressively measurable for every $y\in\BR$. We also adopt the
convention that every c\`adl\`ag process $Y$ on $[0,T]$ extends to
$[0,\infty)$  by $Y_t=Y_T,\, t\ge T$.

\begin{df}
We say that a triple of processes $(Y, M, K)$ is a solution of
reflected backward stochastic differential equation with terminal
condition $\xi$, right-hand side $f+dV$ and lower barrier $L$
(RBSDE$(\xi, f+dV, L)$ for short) if
\begin{enumerate}
\item[(a)]$Y$ is $\FF$ adapted c\`adl\`ag process of Doob's class (D),
$M\in\MM_{0,loc}$, $K\in{ } ^p\mathcal{V}^{+}_0$,
\item[(b)]$Y_{t}\ge L_{t}$ for a.e. $t\in[0,T]$,
\item[(c)]for every c\`adl\`ag process $\hat{L}$ such that
$L_{t}\le \hat{L}_{t}\le Y_{t}$ for a.e. $t\in[0, T]$,
\[
\int_{0}^{T}(Y_{t-}-\hat{L}_{t-})\,dK_{t}=0,
\]
\item[(d)] $[0,T]\ni t\rightarrow f(t,Y_{t})\in L^{1}(0,T)$ and
\[
Y_{t}=\xi+\int_{t}^{T}f(r, Y_{r})\,dr+
\int_{t}^{T}dV_r+\int_{t}^{T}dK_{r}-
\int_{t}^{T}dM_{r}, \quad t\in[0, T].
\]
\end{enumerate}
\end{df}

We shall prove a comparison result (and consequently uniqueness)
for solutions of reflected equations 
under the following monotonicity condition.
\begin{enumerate}
\item[(H1)]There is $\mu\in\BR$ such that for a.e.
$t\in[0, T]$ and every $y,y'\in\BR$,
\[
(f(t, y)-f(t, y'))(y-y')\le\mu|y-y'|^{2}.
\]
\end{enumerate}

\begin{stw}
\label{stw1.1} Let $(Y^{i}, M^{i}, K^{i})$ be a solution of
$\mbox{\rm{RBSDE}}(\xi^{i}, f^{i}+dV^{i}, L^{i})$, $i=1, 2$.
Assume that $\xi^{1}\le\xi^{2}$, $L_{t}^{1}\le L_{t}^{2}$ for a.e.
$t\in[0, T]$, $dV^{1}\le dV^{2}$, and either $f^{1}$ satisfies
$\mbox{\rm{(H1)}}$ and $f^{1}(t, Y^{2}_{t})\le f^{2}(t,
Y^{2}_{t})$ for a.e. $t\in[0, T]$ or $f^{2}$ satisfies
$\mbox{\rm{(H1)}}$ and $f^{1}(t, Y^{1}_{t})\le f^{2}(t,
Y^{1}_{t})$ for a.e. $t\in[0, T]$. Then $ Y_{t}^{1}\le Y_{t}^{2},\, t\in[0, T]$.
\end{stw}

\begin{dow}
Set $Y=Y^{1}-Y^{2}$, $M=M^{1}-M^{2}$, $K=K^{1}-K^{2}$. By the
assumptions and the Tanaka-Meyer formula, for every
$\tau\in\mathcal{T}$ we have
\begin{align}
\label{eq1.1} \nonumber Y_{\tau}^{+}&\le Y_{\tau}^{+}
+\int_{t}^{\tau}\mathbf{1}_{\{Y_{r-}^{1}>Y_{r-}^{2}\}}\,
(f^1(r,Y_r^1)-f^2(r, Y^2_r))\,dr
+\int_{t}^{\tau}\mathbf{1}_{\{Y_{r-}^1>Y^2_{r-}\}}d(V^{1}-V^{2})_{r}\\
&\quad+\int_{t}^{\tau}\mathbf{1}_{\{Y_{r-}^{1}>Y_{r-}^{2}\}}\,dK_r
\nonumber
-\int_{t}^{\tau}\mathbf{1}_{\{Y_{r-}^{1}>Y_{r-}^{2}\}}\,dM_r\\
&\le Y_{\tau}^{+}+\mu\int_{t}^{\tau}Y_{r}^{+}\,dr
+\int_{t}^{\tau}\mathbf{1}_{\{Y_{r-}^{1}>Y_{r-}^{2}\}}\,dK_r^1
-\int_{t}^{\tau}\mathbf{1}_{\{Y_{r-}^{1}>Y_{r-}^{2}\}}\,dM_r.
\end{align}
Observe that $L_{t}^{1}\le Y_{t}^{1}\wedge Y_{t}^{2}\le Y_{t}^{1}$
for a.e. $t\in[0, T]$ and
\[
\int_{t}^{\tau}\mathbf{1}_{\{Y_{r-}^{1}>Y_{r-}^{2}\}}\,dK_r^1=
\int_{t}^{\tau}\mathbf{1}_{\{Y_{r-}^{1}>Y_{r-}^{2}\}}
|Y_{r-}^{1}-Y_{r-}^{2}|^{-1}(Y_{r-}^{1}-Y_{r-}^{1}\wedge Y_{r-}^{2})\,dK_r^1.
\]
Hence
\[
\int_{t}^{\tau}\mathbf{1}_{\{Y_{r-}^{1}>Y_{r-}^{2}\}}\,dK_r^1\le
0, \quad t\in [0, \tau]
\]
by condition (c) of the definition of the solution of
RBSDE$(\xi^{1},f^{1}+dV^{1}, L^{1})$. From (\ref{eq1.1}) it
therefore follows that
\[
Y_{t}^{+}\le Y_{\tau}^{+}+\mu\int_{t}^{\tau} Y_{r}^{+}\,dr-
\int_{t}^{\tau}\mathbf{1}_{\{Y_{r-}>0\}}\,dM_r, \quad t\in[0, \tau].
\]
Let $\{\tau_{n}\}\subset\mathcal{T}$ be a fundamental sequence for
the martingale $M$. Then by the above inequality,
\[
E Y_{t\wedge\tau_{n}}^{+}\le E Y_{\tau_{n}}^{+}+ \mu
E\int_{t\wedge\tau_{n}}^{\tau_n}Y_{r}^{+}\,dr, \quad t\in[0, T].
\]
Since $Y$ is of class (D), letting $n\rightarrow\infty$ gives
\[
EY_{t}^{+}\le \mu\int_{t}^{T}EY_{r}^{+}\,dr, \quad t\in[0,T],
\]
so applying Gronwall's lemma yields the desired result.
\end{dow}

\begin{wn}
\label{wn1.1} Let assumption $\mbox{\rm{(H1)}}$ hold. Then there
exists at most one solution of $\mbox{\rm{RBSDE}}(\xi,f+dV,L)$.
\end{wn}

Let us recall that a filtration $\FF$ is called quasi-left
continuous if  for every sequence $\{\tau_{n}\}\subset\mathcal{T}$
and $\tau\in\mathcal{T}$, if $\tau_{n}\nearrow\tau$ then
$\bigvee_{n\in\BN}\,\FF_{\tau_{n}}=\FF_{\tau}$.

\begin{stw}
\label{stw1.2} Assume that $\FF$ is quasi-left continuous,
$\xi\in L^1 (\FF_{T})$, $V\in\mathcal{V}^{1}_0$ and $L$ is a
c\`adl\`ag process of class $\mbox{\rm{(D)}}$. Set
\begin{equation}
\label{eq2.03} Y_{t}=\esssup_{\tau\in \mathcal{T}_{t}}
E\left(\int_{t}^{\tau}dV_{r}+L_{\tau}\mathbf{1}_{\{\tau<T\}}
+\xi\mathbf{1}_{\{\tau=T\}}|\FF_{t}\right).
\end{equation}
Then for every predictable $\tau\in\mathcal{T}$ such that
$\tau>0$,
\begin{equation}
\label{eq2.0B} Y_{\tau-}=L_{\tau-} \vee (Y_{\tau}+\Delta V_{\tau}).
\end{equation}
\end{stw}
\begin{dow}
For simplicity we assume that $\tau$ is constant. The proof in the
general case goes through as in case $\tau\equiv\mbox{const}$,
with some obvious   changes.

First observe that the process $\bar{Y}_{t}\equiv Y_t
+\int_{0}^{t}dV_{r}$, $t\in[0,T]$, is a supermartingale and if we
put
\[
\bar{L}_t =L_{t}+\int_{0}^{t}dV_{r}, \quad
\bar{\xi}=\xi+\int_{0}^{T}dV_r
\]
then
\[
\bar{Y}_{t-}=\bar{L}_{t-}\vee\bar{Y}_{t} \quad\mbox{iff}\quad
Y_{t-}=L_{t-}\vee(Y_{t}+\Delta V_{t}).
\]
Therefore without loss of generality we may and will assume that
$V=0$. Then $Y$ is a supermartingale. By the assumptions of the
proposition, $Y$ is of class (D). Therefore there exist
$K\in\mathcal{V}^{1,+}_0$ and $M\in\MM_0$ such that
\[
Y_{t}=Y_{0}-K_{t}+M_{t}, \quad t\in[0, T].
\]
Since $\FF$ is quasi-left continuous, $\Delta M_{t}=0$, $P$-a.s.
for every $t\in[0,T]$. Hence  $\Delta Y_{t}\le 0$ for every
$t\in(0,T]$, which implies that for every $t\in(0,T]$,
\begin{equation}
\label{eq1.2} Y_{t}\le Y_{t-}\,.
\end{equation}
On the other hand, $Y_{t}\ge L_{t}$, $t\in[0, T]$. Therefore
$L_{t-}\le Y_{t-}$\,, $t\in(0,T]$. From this and (\ref{eq1.2}) it
follows that for every $t\in(0,T]$,
\[
Y_{t}\vee L_{t-}\le Y_{t-}\,.
\]
We are going to show the opposite inequality. To this end, let us
fix $t\in(0, T]$ and $s\in [0, t)$. By known properties of Snell's
envelope (see \cite{ElKaroui}),
\begin{align*}
Y_{s}&=\esssup_{\tau\in\mathcal{T}_{s, t}} E(L_\tau
\mathbf{1}_{\{\tau<t\}}+Y_{t}\mathbf{1}_{\{\tau=t\}}|\FF_{s})\\
&= E(\esssup_{\tau\in\mathcal{T}_{s, t}} (L_\tau
\mathbf{1}_{\{\tau<t\}}+Y_{t}\mathbf{1}_{\{\tau=t\}})|\FF_{s})\\&
\le E(\esssup_{\tau\in\mathcal{T}_{s, t}} (L_\tau
\mathbf{1}_{\{\tau<t\}}+(Y_{t}\vee
L_{t-})\mathbf{1}_{\{\tau=t\}})|\FF_{s}) \equiv U_{s}.
\end{align*}
Observe that
\begin{equation}
\label{eq1.3} Y_{s}\le U_{s}, \quad s\in(0, t)
\end{equation}
and
\[
U_{s}=E(\esssup_{\tau\in\mathcal{T}_{s, t}} (L_\tau
\mathbf{1}_{\{\tau<t\}}+(Y_{t}\vee L_{t-})
\mathbf{1}_{\{\tau=t\}})|\FF_{s})\ge E(Y_{t}\vee L_{t-}|\FF_{s}),
\]
which implies that
\[
E(Y_{t}\vee L_{t-}|\FF_{t-})\le\liminf_{s\rightarrow t^{-}}U_{s}.
\]
Put
\[
\hat{L}_{r}=L_{r}\mathbf{1}_{\{r<t\}}+(Y_{r}\vee L_{r-})\mathbf{1}_{\{r=t\}},
\quad r\in[0, t].
\]
With this notation,
\[
U_{s}=E(\esssup_{\tau\in\mathcal{T}_{s, t}}\hat{L}_{\tau}|\FF_{s}).
\]
For $\varepsilon>0$ set
\[
B_{\varepsilon}^{s}=\{\sup_{s\le r\le t} \hat{L}_{r}\le\hat{L}_t
+\varepsilon\}.
\]
Since $\hat{L}$ is c\`adl\`ag and has nonnegative jump at $r=t$,
\begin{equation}
\label{eq1.4}
\lim_{s\rightarrow t^{-}}B_{\varepsilon}^{s}=\Omega.
\end{equation}
We have
\begin{align}
\label{eq2.06} E(\esssup_{\tau\in\mathcal{T}_{s,
t}}\hat{L}_{\tau}|\FF_s) &=E(\esssup_{\tau\in\mathcal{T}_{s,
t}}\hat{L}_{\tau} \mathbf{1}_{B^{s}_{\varepsilon}}|\FF_s)
+E(\esssup_{\tau\in\mathcal{T}_{s, t}}\hat{L}_{\tau}
\mathbf{1}_{(B^{s}_{\varepsilon})^c}|\FF_s)\nonumber\\
&\le E((\hat{L}_{t}+\varepsilon)
\mathbf{1}_{B^{s}_{\varepsilon}}|\FF_s)
+E(\esssup_{\tau\in\mathcal{T}_{s, t}}\hat{L}_{\tau}
\mathbf{1}_{(B^{s}_{\varepsilon})^c}|\FF_s).
\end{align}
Since $\hat{L}$ is of class (D), it follows from (\ref{eq1.4})
that
\begin{equation}
\label{eq2.08} E(E(\esssup_{\tau\in\mathcal{T}}|\hat{L}_{\tau}|
\mathbf{1}_{(B^{s}_{\varepsilon})^{\varepsilon}}|\FF_s))=
\sup_{\tau\in\mathcal{T}}E|\hat{L}_{\tau}|
\mathbf{1}_{(B^{s}_{\varepsilon})^c}\rightarrow 0
\end{equation}
as $s\rightarrow t^{-}$. Letting $s\rightarrow t^{-}$ in
(\ref{eq2.06}) and using (\ref{eq1.4}), (\ref{eq2.08}) yields
\[
\limsup_{s\rightarrow t^{-}}U_{s}\le\varepsilon+
E(\hat{L}_{t}|\FF_{t-}).
\]
Since the above inequality holds for  every $\varepsilon>0$ and
the filtration $\FF$ is quasi-left continuous, it follows that
$\limsup_{s\rightarrow t^{-}}U_{s}\le\hat{L}_t$, which when
combined with (\ref{eq1.3}) gives
\[
L_{t-}\vee Y_{t}=\lim_{s\rightarrow t^{-}}U_{s}.
\]
By this and (\ref{eq1.3}),
\[
Y_{t-}=\lim_{s\rightarrow t^{-}}Y_{s}\le\lim_{s\rightarrow
t^{-}}U_{s} =L_{t-}\vee Y_{t},
\]
and the proof is complete.
\end{dow}
\medskip

\begin{prz}
The conclusion of Proposition \ref{stw1.2} does not hold if the
quasi-continuity of filtration is omitted from the hypotheses. To
see this let us consider a random variable $\xi$ such that
$P(\xi=5)=P(\xi=1)=1/2$. Let $T>1$, $\mathcal{G}=\sigma(\xi)$.
Define $\FF=\{\FF_t,\,t\in[0,T]\}$ as follows: $\FF_t=\{\emptyset,
\Omega\}$ for $t\in[0,1)$ and $\FF_t=\mathcal{G}$ for $t\in[1,T]$.
Then the filtration $\FF$ is right-continuous but is not
quasi-left continuous. Let $V=0$ and  $L_t=2$ for $t\in[0, 1)$,
$L_t=0$ for $t\in[1,T]$. Since $Y^n_t\ge Y_t^0\ge L_t$ for
$t\in[0, T]$, it follows that
\[
Y_t^n=Y_t, \quad t\in[0, T],
\]
where $Y$ is defined by (\ref{eq2.03}) and $Y_t^n$ is a solution
of (\ref{eq1.7}) with $\FF,\xi,L$ defined above and $f=V^n=0$.
Moreover,
\[
Y_t= \left\{
\begin{array}{l}E\xi, \quad t\in[0, 1),\smallskip \\
\xi, \quad t\in[1,T],
\end{array}
\right.
\]
from which it follows that $Y_{1-}\neq L_{1-}\vee Y_1$.
\end{prz}

\begin{uw}
\label{uw2.A}  That Proposition \ref{stw1.2} is not true if we
drop the assumption of quasi-left continuity of filtration stems
from the fact that the jumps of $Y$ in predictable times can be
produced by its martingale part.
\end{uw}

Our Proposition \ref{stw1.2} follows from a more general
Proposition \ref{stw2.EK}  proved in \cite{ElKaroui}. We have
decided to provide Proposition \ref{stw1.2} here to make our
presentation self-contained in the important case of quasi-left
continuous filtration. The second reason is that the proof of
Proposition \ref{stw1.2} is much simpler than that of Proposition
\ref{stw2.EK}.

\begin{stw}
\label{stw2.EK} Under the assumptions of Proposition \ref{stw1.2}
but with $\FF$ only satisfying the usual conditions,
\begin{equation}
\label{eq2.B} Y_{t-}=L_{t-}\vee({}^pY_t+{}^pV_t-V_{t-}), \quad
t\in[0,T],
\end{equation}
where ${}^pY$ (resp. ${}^pV$) denotes the predictable projection
of  the process $Y$ (resp. $V$).
\end{stw}
\begin{proof}
See \cite[Proposition 2.34]{ElKaroui}.
\end{proof}
\medskip

In the rest of this section $(Y^{n}, M^{n})$ stands for the
solution of the BSDE
\begin{equation}
\label{eq1.7} Y^{n}_{t}=\xi+\int_{t}^{T}f(r, Y_{r}^{n})\,dr
+\int_{t}^{T}dK^n_r+\int_{t}^{T}dV_{r}^{n}
-\int_{t}^{T}dM_{r}^{n}, \quad t\in[0, T],
\end{equation}
where
\[
K^n_t=\int_0^tn(Y^n_r-L_r)^-\,dr, \quad t\in[0, T]
\]
and $\{V^{n}\}\subset \mathcal{V}_0$ are processes such that
$dV^{n}\le dV^{n+1}$, $n\ge 1$, and $V_{t}^{n}\rightarrow V_{t}$,
$t\in[0, T]$.

\begin{df}
We say that a pair $(Y, M)$ is a supersolution of BSDE$(\xi, f+dV)$
if there exists a process $C\in\mathcal{V}^{1, +}_0$ such that
$(Y, M)$ is a solution of BSDE$(\xi, f+dV+dC)$.
\end{df}
Let us consider the following hypotheses.
\begin{enumerate}
\item[(H2)]$[0, T]\ni t \mapsto f(t,y)\in L^1(0,T)$ for every $y\in\BR$,
\item[(H3)]$\BR\ni y\mapsto f(t, y)$ is
continuous for a.e. $t\in[0, T]$.
\end{enumerate}

\begin{lm}
\label{lm1.1} Assume that $Y_{t}^{n}\nearrow Y_{t}$, $t\in[0, T]$,
$Y$ is a c\`adl\`ag process of class $\mbox{\rm{(D)}}$,
$f_{Y}\equiv f(\cdot, Y)\in L^1 (\FF)$, $V\in\mathcal{V}^1_0$ and
$\mbox{\rm{(H1)--(H3)}}$ are satisfied. Then $Y$ is the smallest
supersolution of  $\mbox{\rm{BSDE}}(\xi, f_{Y}+dV)$ such that
$Y_{t}\ge L_{t}$ for a.e. $t\in[0, T]$.
\end{lm}
\begin{dow}
Put
\[
\tilde{Y}_{t}^{n}=Y_t^n+\int_{0}^{t}f(r, Y^{n}_r)\,dr+
\int_{0}^{t}dV_{r}^n, \quad t\in[0, T].
\]
Observe that $\tilde{Y}^{n}$ is a supermartingale of class (D) and
that by $\mbox{\rm{(H1)--(H3)}}$, $\tilde{Y}_t^n
\rightarrow\tilde{Y}_t$, $t\in[0, T]$, where
\[
\tilde{Y}_{t}=Y_t+\int_{0}^{t}f(r, Y_{r})\,dr +\int_{0}^{t}dV_{r},
\quad t\in[0, T].
\]
Since $\tilde{Y}$ is also a supermartingale of class (D), there
exist $M\in\MM_0$ and $K\in\mathcal{V}^{1,+}_0$ such that
\[
\tilde{Y}_{t}=\tilde{Y}_{0}+K_{t}+M_{t}, \quad t\in[0,T].
\]
Hence
\[
Y_t=\xi+\int_{t}^{T}f(r, Y_{r})\,dr+\int_{t}^{T}dV_{r}+
\int_{t}^{T}dK_{r}-\int_{t}^{T}dM_{r}, \quad t\in[0, T],
\]
i.e. $(Y, M)$ is a supersolution of BSDE$(\xi, f_{Y}+dV)$. Let
$(\hat{Y}, \hat{M})$ be a supersolution of BSDE$(\xi, f_{Y}+dV)$
such that $L_{t}\le \hat{Y}_t$ for a.e. $t\in[0, T]$. By the
definition of a supersolution there exists $C\in\mathcal{V}^{1,
+}_0$ such that
\[
\hat{Y}_t=\xi+\int_{t}^{T}f(r, Y_{r})\,dr+\int_{t}^{T}dV_{r}
+\int_{t}^{T}dC_{r}-\int_{t}^{T}d\hat{M}_{r}, \quad t\in[0, T].
\]
Let $(\bar{Y}^{n}, \bar{M}^{n})$ be a solution of the BSDE
\[
\bar{Y}_t^n=\xi+\int_{t}^{T}f(r, Y_{r})\,dr+\int_{t}^{T}dV_{r}
+\int_{t}^{T}n(\bar{Y}_r^n -L_{r})^-\,dr
-\int_{t}^{T}d\bar{M}_{r}^n, \quad t\in[0,T].
\]
Since
\begin{align*}
Y_t=\xi&+\int_{t}^{T}f(r, Y_{r})\,dr+\int_{t}^{T}dV_{r}+
\int_{t}^{T}n(Y_{r}-L_r)^- \,dr+\int_{t}^{T}dV_{r}\\&+
\int_{t}^{T}dK_{r}-\int_{t}^{T}dM_{r}, \quad t\in[0, T],
\end{align*}
it follows from \cite[Proposition 2.1]{KR:JFA} that
$\bar{Y}^n_t\le Y_t$, $t\in[0,T]$. Arguing as in the case of the
sequence $\{Y^{n}\}$ we show that there exist
$\bar{K}\in\mathcal{V}^{1, +}_0$ and a c\`adl\`ag process
$\bar{Y}$ of class (D) (since $\bar{Y}^n\le Y$) such that
\[
\bar{Y}_t=\xi+\int_{t}^{T}f(r, Y_{r})\,dr+\int_{t}^{T}d\bar{K}_{r}
+\int_{t}^{T}dV_r-\int_{t}^{T}d\bar{M}_r, \quad t\in[0,T]
\]
and $\bar{Y}_t^n \nearrow\bar{Y}_t$ for every $t\in[0, T]$. By
\cite[Proposition 2.1]{KR:JFA}, $\bar{Y}_t^n \le\hat{Y}_t$,
$t\in[0, T]$, because
\[
\hat{Y}_t=\xi+\int_{t}^{T}f(r, Y_{r})\,dr+\int_{t}^{T}dV_{r}
+\int_{t}^{T}n(\hat{Y}_r -L_r)^-dr+\int_{t}^{T}dC_{r}-
\int_{t}^{T}d\hat{M}_{r}, \quad t\in[0, T].
\]
Thus
\begin{equation}
\label{eq1.8}
\bar{Y}_t \le \hat{Y}_t, \quad t\in[0, T].
\end{equation}
Now we will show that $\bar{Y}_t=Y_t$, $t\in[0, T]$. To this end,
let us set
\[
\bar{K}_t^n= \int_{0}^{t} n(\bar{Y}_r^n-L_r)^-\,dr, \quad
U^n_{t-}=\widehat{\mbox{sgn}}(Y^n_{t-}-\bar{Y}^n_{t-}),
\]
where $\hat x=\frac{x}{|x|}\mathbf{1}_{\{x\neq0\}}$. By the
Tanaka-Meyer formula, for every $\tau\in\mathcal{T}$ we have
\begin{align*}
|Y_t^n-\bar{Y}_t^n|&\le |Y_{\tau}^{n}-\bar{Y}_{\tau}^n|+
\int_{t}^{\tau}U_{r-}^{n}\,d(K^n-\bar{K}^n)_{r}+
\int_{t}^{\tau}U_{r-}^{n}(f(r, Y_{r}^{n})-f(r, Y_r))\,dr\\
&\quad- \int_{t}^{\tau}U_{r-}^n\,d(M^n-\bar{M}^n)_r \le
|Y_{\tau}^n-\bar{Y}_\tau^n|+ \int_{t}^{\tau}|f(r, Y_{r}^{n})-f(r,Y_r)|\,dr\\
&\quad- \int_{t}^{\tau}U_{r-}^n\,d(M^n-\bar{M}^n)_r, \quad
t\in[0,\tau].
\end{align*}
Let $\{\tau_{k}\}\subset\mathcal{T}$ be a fundamental sequence for
$M^n-\bar{M}^n$. Since $Y^n-\bar{Y}^n$ is of class (D), replacing
$\tau$ by $\tau_{k}$ in the above inequality, taking the
expectations and then letting $k\rightarrow\infty$  we get
\begin{equation}
\label{eq2.07} E|Y^n_t-\bar{Y}^n_t|\le
E\int_{t\wedge\tau}^{\tau}|f(r,Y_{r}^{n})-f(r, Y_r)|\,dr, \quad
t\in[0,T].
\end{equation}
By (H1),
\[
f(r, Y_r)-\mu Y_r\le f(r, Y_r^n)-\mu Y_r^n \le f(r,
Y_r^1)-\mu Y_r^1,
\]
so applying the Lebesgue dominated convergence theorem shows that
the right-hand side of (\ref{eq2.07}) tends to zero. Therefore
$\bar{Y}_t =Y_t$, $t\in[0,T]$, which when combined with
(\ref{eq1.8}) implies that $Y_t \le\hat{Y}_t$, $t\in[0,T]$.
\end{dow}

\begin{wn}
\label{wn1.2} Under the assumptions of Lemma \ref{lm1.1},
\[
Y_{t}=\esssup_{\tau\in\mathcal{T}_{t}}E\Big(\int_{t}^{\tau}f(r,Y_r)\,dr
+\int_{t}^{\tau}dV_r+\hat{L}_\tau\mathbf{1}_{\{\tau<T\}}
+\xi\mathbf{1}_{\{\tau<T\}}|\FF_t\Big)
\]
for every c\`adl\`ag process $\hat{L}$ such that
$L_t\le\hat{L}_t\le Y_t$ for a.e. $t\in[0,T]$.
\end{wn}
\begin{dow}
From Lemma \ref{lm1.1} it is clear  that the process $Y$ is the
smallest supersolution of BSDE$(\xi, f_{Y}+dV)$ such that $Y_t\ge
\hat{L}_t$, $t\in[0, T]$. From this we conclude that $\tilde{Y}$
defined in the proof of Lemma \ref{lm1.1} is the smallest
supermartingale with the property that
$\tilde{Y}_T=\tilde{\xi}\equiv\xi+\int_{0}^{T}f(r,Y_r)\,dr
+\int_{0}^{T}dV_r$ majorizing the process $\tilde{L}_t=\hat{L}_t
+\int_{0}^{t}f(r,Y_r)\,dr+\int_{0}^{t}dV_r$. Therefore
\[
\tilde{Y}_t=\esssup_{\tau\in\mathcal{T}_{t}} E(\tilde{L}_\tau
\mathbf{1}_{\{\tau<T\}}
+\tilde{\xi}\mathbf{1}_{\{\tau=T\}}|\FF_t),
\]
from which the desired result immediately follows.
\end{dow}

\begin{wn}
\label{wn1.3} Under the assumptions of Lemma \ref{lm1.1},
\[
Y_{t-}=\hat{L}_{t-}\vee(^pY_t+{ }^pV_t-V_{t-}), \quad t\in[0, T].
\]
for any c\`adl\`ag process $\hat{L}$  such that
$L_t\le\hat{L}_t\le Y_t$ for a.e. $t\in[0,T]$.
\end{wn}
\begin{dow}
Follows immediately from Corollary \ref{wn1.2} and Proposition
\ref{stw2.EK}.
\end{dow}

\begin{lm}
\label{lm1.2} Assume that $Y$ is a supermartingale of class
$\mathcal{S}^{2}$ admitting the decomposition
\[
Y_t=Y_0-K_t+M_t, \quad t\in[0,T]
\]
for some $K\in\, ^p\mathcal{V}_0^+$, $M\in\MM_0$. Then there exists
$c>0$ (not depending on $Y$) such that
\[
E[Y]_{T}+E|K_T|^2 \le cE\sup_{t\le T}|Y_t|^2.
\]
\end{lm}
\begin{dow}
Without loss of generality we may assume that $Y$ is bounded from
above, for  otherwise we can first prove the inequality for the
supermartingale $Y\wedge n$ and then pass with $n$ to the limit.
By the It\^o-Meyer formula, for every $\tau\in\mathcal{T}$,
\begin{equation}
\label{eq1.9}
|Y_t|^2=|Y_\tau|^2+\int_{t\wedge\tau}^{\tau}Y_{r-}\,dK_r
-\int_{t\wedge\tau}^{\tau}Y_{r-}\,dM_r-\int_{t\wedge\tau}^{\tau}d[Y]_r,
\quad t\in[0, \tau].
\end{equation}
We also have
\begin{equation}
\label{eq1.10}
Y_t=Y_\tau+\int_{t\wedge\tau}^{\tau}dK_r
-\int_{t\wedge\tau}^{\tau}dM_r,
\quad t\in[0, \tau].
\end{equation}
Let $\{\tau_k\}\subset\mathcal{T}$ be a fundamental sequence for
$M$. Then replacing $\tau$ by $\tau^k$ in the above equation,
taking the expectation and then letting $k\rightarrow \infty$ we
get
\begin{equation}
\label{eq2.01} EK_T=EY_0-EY_T<\infty.
\end{equation}
Using the localization procedure one can deduce from (\ref{eq1.9})
and (\ref{eq2.01}) that
\[
EY_t^2+E\int_{t}^{T}d[Y]_r\le E|Y_T|^2+\|Y\|_\infty EK_T<\infty.
\]
Let us recall that since $K$ is predictable $E[Y]_T=E[M]_T+E[K]_T$.
Squaring both sides of (\ref{eq1.10}), applying Doob's
$L^2$-inequality and performing  standard calculations we conclude
that there exists $c_1>0$ such that
\begin{equation}
\label{eq2.012} E|K_T|^2\le c_1(E\sup_{t\le T}|Y_t|^2+E[Y]_T).
\end{equation}
By (\ref{eq1.9}) and  Doob's $L^2$-inequality,
\[
E\sup_{t\le T}|Y_t|^2+E[Y]_T\le E\sup_{t\le T}|Y_t|^2+\alpha
E\sup_{t\le T} |Y_t|^2+\frac{1}{\alpha} E|K_T|^2+4\alpha
E\sup_{t\le T}|Y_t|^2+\frac{1}{\alpha} E[Y]_T
\]
for every $\alpha>0$. The lemma follows from (\ref{eq2.012}) and
the above inequality with $\alpha=2c_1+2$.
\end{dow}

\begin{lm}
\label{lm2.A} Assume that
\begin{equation}
Y_t^n=Y_0^n-A_t^n+M_t^n, \quad t\in[0, T],
\end{equation}
where $A^n\in{ }^p\VV^2_0$, $M^n\in\MM_0^2$, $Y^1, Y\in \mathcal{S}^2$ and
$Y_t^n\le Y^{n+1}_t$, $t\in[0, T]$, $n\in\BN$. Then the process
$Y$ defined as  $Y_t=\lim_{n\rightarrow\infty} Y^n_{t}$, $t\in[0,
T]$, is c\`adl\`ag.
\end{lm}
\begin{dow}
By Lemma \ref{lm1.2},
\[
E|A^n|_T^2+E[M^n]_T\le c E\sup_{t\in[0, T]}|Y^n_t|^2\le
cE\sup_{t\in[0, T]}(|Y^1_t|^2\vee|Y_t|^2).
\]
It follows in  particular that $\sup_nE|M^n_T|^2<\infty$.
Therefore there is $X\in L^2(\FF_T)$ such that $M^n_T\rightarrow
X$ weakly in $L^2(\FF_T)$. Let $N$ be a c\`adl\`ag version of the
martingale $E(X|\FF_t)$, $t\in[0, T]$. Then  for every
$\tau\in\mathcal{T}$, $M^n_\tau\rightarrow N_\tau$ weakly in
$L^2(\FF_T)$. Indeed, for any $Z\in L^2(\FF_T)$,
\begin{align*}
EM^n_\tau Z=E(E(M^n_T|\FF_\tau) Z)&=E(M^n_TE(Z|\FF_\tau))\\
&\quad\rightarrow E(XE(Z|\FF_\tau))=E(E(X|\FF_\tau)Z)=EN_\tau Z.
\end{align*}
Put
\[
A_t=Y_0-Y_t+N_{t}, \quad t\in[0, T].
\]
Then for every $\tau\in\mathcal{T}$,
\[
A^n_{\tau}=Y_0^n-Y^n_{\tau}+M_\tau^n\rightarrow
Y_0-Y_\tau+N_\tau=A_\tau
\]
weakly in $L^2(\FF_T)$. From the above convergence we conclude
that $A_\sigma\le A_\tau$ for every $\sigma, \tau\in\mathcal{T}$
such that $\sigma\le\tau$. From the section theorem it now follows
that $A$ is an increasing process. Consequently, $Y$ is c\`adl\`ag
by \cite[Lemma 2.2]{Peng}.
\end{dow}
\medskip

We will need the following hypotheses.
\begin{enumerate}
\item[(H4)]$\xi\in L^1(\FF_T)$, $V\in\mathcal{V}_0^1$, $f(\cdot, 0)\in
L^1(\FF)$.
\item[(H5)] There exists $X\in \mathcal{V}^1 \oplus\MM$ such that
$X_t\ge L_t$ for a.e. $t\in[0, T]$ and $f^-(\cdot, X)\in
L^1(\FF)$.
\end{enumerate}

Let $s\in[0,T)$ and let $\tau\in\TT_t$. We will say that a
sequence of processes $\{X^n\}$ converges to $X$ uniformly in
probability on $[s,\tau)$ (ucp on $[s,\tau)$ for short) if for
every subsequence $\{n'\}$ there is a further subseqence $\{n''\}$
such that $X^{n''}\rightarrow X$ a.s. uniformly on compact subsets
of $[s,\tau)$.

\begin{tw}
\label{tw1.1} Assume that $\mbox{\rm{(H1)--(H4)}}$ hold. Then
there exists a solution $(Y, M, K)$ of $\mbox{\rm{RBSDE}}(\xi,
f+dV, L)$ such that $K\in\mathcal{V}^{1,+}$ iff $\mbox{\rm{(H5)}}$
is satisfied. Moreover, under $\mbox{\rm{(H1)--(H5)}}$,
$Y^n_t\nearrow Y_t$, $t\in[0, T]$, $Y^n\rightarrow Y$ and
$\int_{s}^\cdot dK_r^n \rightarrow\int_{s}^\cdot dK_r$ in ucp on
$[s,\tau_{s})$ for every $s\in [0,T)$, where $\tau_s=\inf\{t>s;
\Delta K_t>0\}$, and finally, for every $\tau\in\mathcal{T}$,
$EK_\tau^n\rightarrow EK_\tau$.
\end{tw}
\begin{dow}
By \cite[Lemma 2.3]{KR:JFA}, if there exists a solution of
RBSDE$(\xi, f+dV,L)$ such that $K\in\mathcal{V}^{1,+}$ then (H5)
is satisfied with $X=Y$. Suppose now that (H5) is satisfied. By
\cite[Proposition 2.1]{KR:JFA}, for every $n\ge0$,
\begin{equation}
\label{eq1.11}
Y^n_t\le Y_t^{n+1},\quad t\in[0, T].
\end{equation}
By (H5) there exists $X$ of class (D) such that
\begin{equation}
\label{eq1.12} X_t\ge L_t \mbox{ for a.e. }t\in[0,T], \quad
f^-(\cdot, X_\cdot)\in L^1(\FF)
\end{equation}
and
\[
X_t=X_0-U_t+N_t, \quad t\in[0, T]
\]
for some $N\in\mathcal{M}_0$, $U\in\mathcal{V}_0^1$. Clearly,
\begin{align*}
X_t&=X_T+\int_{t}^{T}f(r, X_r)\,dr-\int_{t}^{T}f^+(r, X_r)
+\int_{t}^{T}f^-(r, X_r)\,dr\\
&\quad + \int_{t}^{T}dU_r-\int_{t}^{T}dN_r, \quad t\in[0,T].
\end{align*}
Let $(\bar{X}, \bar{N})$ be a solution of the BSDE
\begin{align*}
\bar{X}_t&=X_T \vee\xi+\int_{t}^{T}f(r,
\bar{X}_r)\,dr+\int_{t}^{T}f^-(r, X_r)\,dr\\
&\quad+ \int_{t}^{T}dV_r^+
\int_{t}^{T}dU^+_r-\int_{t}^{T}d\bar{N}_r, \quad t\in[0, T].
\end{align*}
Then by \cite[Proposition 2.1]{KR:JFA}, $\bar{X}_t\ge X_t$,
$t\in[0,T]$, and hence, by (\ref{eq1.12}), $\bar{X}_t\ge L_t$ for
a.e. $t\in[0,T]$. Therefore
\begin{align*}
\bar{X}_t&=X_T\vee\xi+\int_{t}^{T}f(r, \bar{X}_r)\,dr+
\int_{t}^{T}n(\bar{X}_r-L_r)^-\,dr\\
&\quad+\int_{t}^{T}f^-(r, X_r)\,dr
\int_{t}^{T}dV^+_r+\int_{t}^{T}dU_r^+-\int_{t}^{T}d\bar{N}_r,
\quad t\in[0,T],
\end{align*}
so using (\ref{eq1.11}) and once again \cite[Proposition
2.1]{KR:JFA} gives
\begin{equation}
\label{eq1.13} Y_t^0\le Y_t^n\le\bar{X}_t, \quad t\in[0,T].
\end{equation}
By \cite[Lemma 2.3]{KR:JFA},
\begin{equation}
\label{eq1.14}
E\int_{0}^{T}|f(r, \bar{X}_r)|\,dr+E\int_{0}^{T}|f(r, Y_r^0)|\,dr<\infty.
\end{equation}
Write $\varphi(x)=x/(1+x)$, $x\ge0$ and $\bar{Y}_t^n\equiv
Y_t^n-Y_t^0+V_t^n-V_t^0$, $t\in[0,T]$, and observe that by
(\ref{eq1.11}), $\bar{Y}_t^n\ge 0$ for $t\in[0,T]$. By the
It\^o-Meyer formula,
\begin{align*}
\varphi(\bar{Y}_t^n)&=\varphi(\bar{Y}_0^n)
+\int_{0}^{t}\varphi'(\bar{Y}_{r-}^n)\,d\bar{Y}_r^n+ \frac12
\int_{0}^{t}\varphi''(\bar{Y}_{r-}^n)\,d[\bar{Y}^n]_r^c\\
&\quad+\sum_{0<s\le t}(\Delta\varphi(\bar{Y}_s^n)
-\varphi'(\bar{Y}^n_{s-})\Delta\bar{Y}_s^n), \quad t\in[0, T].
\end{align*}
For $t\in[0, T]$ set
\[
A_t^n=\int_{0}^{t}\varphi'(\bar{Y}^n_{r-})\,dK_r^n, \quad
B_t^n=-\frac12
\int_{0}^{t}\varphi''(\bar{Y}^n_{r-})\,d[\bar{Y}^n]_r^c,
\]
\[
C_t^n=-\sum_{0<s\le t}(\Delta\varphi(\bar{Y}_s^n)
-\varphi'(\bar{Y}^n_{s-})\Delta\bar{Y}_s^n),\quad \tilde{C}^n_t=(C_t^n)^p,
\]
\[
F_t^n=\int_{0}^{t}\varphi'(\bar{Y}_{r-}^n)(f(r,Y_r^n)-f(r,Y^0_r))\,dr,
\]
\[
N_t^n=\int_{0}^{t}\varphi'(\bar{Y}_r^n)\,d(M^n-M^0)_r+(\tilde{C}_t^n-C^n_t), \quad
Z_t^n=\varphi(\bar{Y}^n_t)-L_t^n.
\]
Then
\[
Z_t^n=Z_0^n-A_t^n-B_t^n-\tilde{C}_t^n+N_t^n, \quad t\in[0, T].
\]
Since $\varphi$ is nondecreasing and concave, $A^n, B^n,
C^n\in\mathcal{V}^+_0$. By (H1) and (\ref{eq1.11}),
\[
f(r, Y_r)-f(r, Y_r^0)-\mu(Y^n_r-Y^0_r)\le
f(r, Y^n_r)-f(r, Y^0_r)\le\mu(Y_r^n-Y^0_r)
\]
for a.e. $r\in[0,T]$. Hence
\begin{equation}
\label{eq1.14a} |f(r, Y^n_r)-f(r, Y^0_r)|\le 2|\mu||Y_r-Y^0_r|
+|f(r,Y_r)-f(r,Y^0_r)|
\end{equation}
for a.e. $r\in[0, T]$. By (\ref{eq1.11}), (\ref{eq1.13}),
(\ref{eq1.14}) and (H3),
\begin{equation}
\label{eq1.14b}
E\int_0^T|f(r, Y_r)|<\infty,
\end{equation}
where
\begin{equation}
\label{eq1.14c} Y_t=\sup_{n\ge 0}Y^n_t, \quad t\in[0,T].
\end{equation}
Since $\varphi'$ is bounded, it follows from (\ref{eq1.14a}) that
there exists a stationary sequence $\{\sigma_k\}\in\mathcal{T}$
(i.e. $P(\liminf_{k\rightarrow\infty}\{\sigma_k=T\})=1$) such that
$\sup_{n\ge1}E|F^n|_{\sigma_k}^2<\infty$. By Lemma \ref{lm1.2}
there exists $c>0$ not depending of $n$ such that
\[
E|A_{\sigma_k}^n|^2+E|B_{\sigma_k}^n|^2+E|\tilde{C}_{\sigma_k}^n|^2
+E[Z^n]_{\sigma_k}\le c.
\]
By (\ref{eq1.11}) and our assumptions on $\{V_n\}$, $Z_t^n\le
Z_t^{n+1}$, $t\in[0,T]$, $n\ge0$, and
\[
Z_t^n\nearrow Z_t, \quad t\in[0,T],
\]
where $Z_t=\varphi(\bar{Y}_t)-F_t$,
$\bar{Y}_t=Y_t+V_t-Y_t^0-V_t^0$, and
\begin{equation}
\label{eq1.15}
F_t=\lim_{n\rightarrow\infty}F_t^n,
\quad t\in[0, T].
\end{equation}
By (\ref{eq1.14a})--(\ref{eq1.14c}) and (H3),
\begin{equation}
\label{eq1.16}
\int_{0}^{t}f(r, Y_r^n)\,dr\rightarrow \int_{0}^{t}f(r, Y_r)\,dr,
\quad t\in[0, T].
\end{equation}
Since $\varphi'$ is  bounded and continuous, it follows from
(\ref{eq1.16}) that
\begin{equation}
\label{eq1.17}
F_t=\int_0^t\varphi'(\bar{Y}_r)(f(r, Y_r)-f(r, Y^0_r))\,dr, \quad t\in[0, T].
\end{equation}
By Lemma \ref{lm2.A}, $Z$ is a c\`adl\`ag process, so $Y$ is
c\`adl\`ag, too. Set
\[
S_t=Y_t+\int_0^tf(r, Y_r)\,dr+\int_0^tdV_r, \quad t\in[0, T].
\]
Observe that $S^n$ defined as
\[
S_t^n=Y_t^n+\int_{0}^{t}f(r, Y_r^n)\,dr+ \int_{0}^{T}dV_r^n, \quad
t\in[0,T]
\]
is a supermartingale and by (\ref{eq1.14c}), (\ref{eq1.16}) and
the assumptions on $\{V^n\}$, $S_t^n\rightarrow S_t$, $t\in[0,T]$.
From the last convergence and (\ref{eq1.13}), (\ref{eq1.17}) it
follows  that $S$ is a c\`adl\`ag supermartingale of class (D).
Therefore there exist $K\in{}^p\mathcal{V}^{1,+}_0$ and
$M\in\MM_0$ such that
\begin{equation}
\label{eq1.17a}
Y_t=\xi+\int_t^Tf(r, Y_r)\,dr+\int_t^TdK_r+\int_t^TdV_r-
\int_t^TdM_r, \quad t\in[0, T].
\end{equation}
Let
\[
R^n_t=Y_t^n+V_t^n,\quad R_t=Y_t+V_t,\quad t\in[0,T].
\]
Then for $t\in[0,T]$,
\[
R^n_t=R^n_0-\int_0^tf(r, Y^n_r)\,dr-K^n_t+M_t^n, \quad
{}^pR^n_t=R^n_0-\int_0^tf(r, Y^n_r)\,dr-K^n_t+M_{t-}^n\,,
\]
\[
R_t=R_0-\int_0^tf(r, Y_r)\,dr-K_t+M_t,
\]
and since $K$ is predictable,
\[
{}^pR_t=R_0-\int_0^tf(r, Y_r)\,dr-K_t+M_{t-}\,.
\]
Since $R_t^n\nearrow R_t$, $t\in[0, T]$, it follows that
$^pR_t^n\nearrow { }^pR_t$, $t\in[0, T]$. Let us fix $s\in [0,T)$. Observe that
\[
{}^pR_t^n=R_{t-}^n, \quad t\in[0,T], \quad {}^pR_t=R_{t-}, \quad
t\in(s,\tau_s).
\]
By Dini's theorem, $R^n\rightarrow R$ in ucp on $[s,\tau_s)$.
Since by the assumption $V^n\rightarrow V$ uniformly on $[0,T]$,
$Y^n\rightarrow Y$ in ucp on $[s,\tau_s)$. Let
\[
R_t^{s, n}=Y_t^{n, (s)}+V_t^{n, (s)}, \quad
R_t^s=Y_t^{(s)}+V_t^{(s)}, \quad t\in[0, T]
\]
with the notation $W^{(s)}_t=W_{t\vee s}-W_{s}$. Then
\[
R_t^{s, n}=-\int_{s}^{t\vee s}f(r, Y^n_r)\,dr-
\int_{s}^{t\vee s}dK_r^n+
\int_{s}^{t\vee s}dM_r^n, \quad t\in[0, T]
\]
and
\[
R_t^{s}=-\int_{s}^{t\vee s}f(r, Y_r)\,dr-
\int_{s}^{t\vee s}dK_r+
\int_{s}^{t\vee s}dM_r, \quad t\in[0, T].
\]
Let $\{T^s_m,\ m\ge 1\}$ be an announcing sequence for $\tau_s$
($\tau_s$ is predictable since $K$ is predictable). Then by what
has already been proved,
\begin{equation}
\label{eq1.18} \sup_{t\in[0, s\vee T^s_m]}|Y_t^{n,
(s)}-Y_t^{(s)}|\rightarrow0, \quad \sup_{t\in[0, s\vee
T^s_m]}|R_t^{s, n}-R_t^s|\rightarrow0.
\end{equation}
By the assumptions on $\{V_n\}$ and (\ref{eq1.13}),
\begin{equation}
\label{eq1.19} \sup_{n\ge1}EK_T^n<\infty.
\end{equation}
Therefore by \cite[Proposition 1-5]{MeminSlominski} the sequence
$\{R^{s, n}-R^s\}$ satisfies the condition UT. Therefore applying
the results of \cite{Jacod} we obtain
\[
[R^{s, n}-R^s]_{T^s_m\vee s}=[M^{n,(k)}-M^{(k)}]_{T^s_m\vee s}
\rightarrow_P 0
\]
as $n\rightarrow\infty$, or, equivalently,
\[
\sup_{t\in[0, s\vee T^k_m]}|M_t^{n, (s)}- M_t^{(s)}|\rightarrow_P 0.
\]
From this we conclude that for every  $m\ge 1$,
\begin{equation}
\label{eq1.20} \sup_{t\in[s, s\vee T^s_m]}\Big|\int_{s}^t
dK_r^n-\int_{s}^tdK_r\Big|\rightarrow_P 0.
\end{equation}
Observe that by (\ref{eq1.19}), $Y_t\ge L_t$ for a.e. $t\in[0,T]$.
Let $\hat{L}$ be a c\`adl\`ag process such that
$L_t\le\hat{L}_t\le Y_t$ for a.e. $t\in[0, T]$. Then
\begin{equation}
\label{eq1.21}
\int_{s}^{s\vee T_m^s}(Y_{t-}^n-\hat{L}_{t-})\,dK_t^n\le
n\int_{s}^{s\vee T_m^s}(Y_t^n-\hat{L}_t)(Y_t^n-\hat{L}_t)^-\,dt\le 0.
\end{equation}
By (\ref{eq1.18}) and (\ref{eq1.20}),
\[
\int_{s}^{s\vee T_m^s}(Y_t^n-\hat{L}_t)\,dK_t^n
\rightarrow \int_{s}^{s\vee T_m^s}(Y_t-\hat{L}_t)\,dK_t
=\int_{s}^{s\vee T_m^s}(Y_{t-}-\hat{L}_{t-})\,dK_t
\]
since $\Delta K_t=0$ for $t\in(s,\tau_{s})$. Since
$Y_{t-}\ge \hat{L}_{t-}$ for $t\in(0, T]$, it follows from
(\ref{eq1.21}) that for every $s\in [0,T)$,
\begin{equation}
\label{eq1.22}
\int_{s}^{\tau_{s}}(Y_{t-}-\hat{L}_{t-})\,dK_t=0.
\end{equation}
By the definition of  $\{\tau_s\}$, $\Delta
K_{\tau_s}>0$ on $\{\tau_s<\infty\}$. Hence
\[
^pY_{\tau_s}+ { }^pV_{\tau_s}-V_{\tau_s-}<Y_{\tau_s-}\quad
\mbox{on}\quad \{\tau_s<\infty\}
\]
since $^pY_{\tau_s}-Y_{\tau_s-}=-\Delta
K_{\tau_s}-(^pV_{\tau_s}-V_{\tau_s-})$. Therefore by Corollary
\ref{wn1.3},  $Y_{\tau_s-}=\hat{L}_{\tau_s-}$ on
$\{\tau_s<\infty\}$, which when  combined with (\ref{eq1.22})
shows that
\[
\int_0^T(Y_{t-}-\hat{L}_{t-})\,dK_t=0.
\]
Integrability of $K$ follows from (\ref{eq1.17}) and
the fact that $Y$ is of class (D).
\end{dow}

\begin{stw}
\label{stw1.3} Assume that $(\xi^i, f^i+dV^i, L)$, $i=1, 2$,
satisfy $\mbox{\rm{(H1)--(H5)}}$. Let $(Y^i, M^i, K^i)$ be a
solution of $\mbox{\rm{RBSDE}}(\xi^i, f^i+dV^i, L)$, $i=1,2$. If
$\xi^1\le\xi^2$, $dV^1\le dV^2$ and either $f^1(t, Y_t^1)\le
f^2(t,Y_t^2)$ or $f^1(t, Y_t^2)\le f^2(t, Y_t^2)$ for a.e.
$t\in[0, T]$ then $dK^2\le dK^1$.
\end{stw}
\begin{dow}
By \cite[Proposition 2.1]{KR:JFA},
\begin{equation}
\label{eq1.23} Y_t^{1, n}\le Y_t^{2,n}, \quad t\in[0,T],
\end{equation}
where $Y^{i, n}$ is a solution of (\ref{eq1.7}) with $(\xi, f, V)$
replaced by $(\xi^i,f^i,V^i)$, $i=1,2$. For $s\in [0,T)$ we set
$\tau_s=\inf\{t>s; \Delta K^1_t+\Delta K^2_t>0\}$. By Theorem
\ref{tw1.1},
\begin{equation}
\label{eq1.24}
\int_{s}^tdK_r^{n,i}\rightarrow\int_{s}^tdK^i, \quad
t\in[s, \tau_{s}),
\end{equation}
for every $s\in [0,T)$, where $K_t^{n,
i}=\int_0^tn(Y_r^{n,i}-L_r)^-\,dr$. From this and (\ref{eq1.23})
we conclude that for every $s\in [0,T)$,
\begin{equation}
\label{eq1.25} dK^2\le dK^1 \quad \mbox{on }\,[s,\tau_{s}).
\end{equation}
Let $\hat{L}$ be a c\`adl\`ag process such that
$L_t\le\hat{L}_t\le Y_t^i$ for a.e. $t\in[0, T]$, $i=1,2$ (for
instance one can take $\hat{L}=Y^1\wedge Y^2$). By the definition
of a solution of RBSDE$(\xi^2, f^2+dV^2, L)$, if $\Delta
K^2_{\tau_s}>0$ then $Y^2_{\tau_s-}=\hat{L}_{\tau_s-}$ on
$\{\tau_s<\infty\}$. Therefore by the assumptions of the
proposition and (\ref{eq1.23}), if $\Delta K^2_{\tau_s}>0$ and
$\tau_s<\infty$ then
\begin{align*}
\Delta K^1_{\tau_s}&=-(^pY^1_{\tau_s}-Y^1_{\tau_s-})
-(^pV^1_{\tau_s}-V^1_{\tau_s-})\ge
-(^pY^1_{\tau_s}-\hat{L}_{\tau_s-})-(^pV^1_{\tau_s}-V^1_{\tau_s-})\\
&\ge-(^pY^2_{\tau_s}-\hat{L}_{\tau_s-})-(^pV^2_{\tau_s}-V^2_{\tau_s-})
=-(^pY^2_{\tau_s}-Y^2_{\tau_s-})-(^pV^2_{\tau_s}-V^2_{\tau_s-})=\Delta
K_{\tau_s}^2.
\end{align*}
This and (\ref{eq1.25}) proves the proposition.
\end{dow}

\nsubsection{BSDEs with two reflecting barriers: the case $p=1$}
\label{sec3}

In this section we prove the existence of solutions of BSDEs with
two reflecting barriers $L,U$ which are separated by a
semimartingale. In what follows the upper barrier $U$  is a
progressively measurable process such that $L_t\le U_t$ for a.e.
$t\in[0, T]$.

\begin{df}
We say that a triple $(Y,M,A)$ is a solution of the reflected BSDE
with terminal condition $\xi$, right-hand side $f+dV$ and upper
barrier $U$ (RBSDE$(\xi, f+dV, U)$ for short) if
\begin{enumerate}
\item[(a)] $Y$ is an $\FF$ adapted c\`adl\`ag process  of class
(D), $M\in\MM_{0,loc}$, $A\in{ }^{p}\mathcal{V}^{+}_0$,
\item[(b)]$Y_{t}\le U_{t}$ for a.e. $t\in[0, T]$,
\item[(c)] for every c\`adl\`ag process $\hat{U}$ such that
$U_{t}\ge \hat{U}_{t}\ge Y_{t}$ for a.e. $t\in[0, T]$,
\[
\int_{0}^{T}(\hat{U}_{t-}-Y_{t-})\,dA_{t}=0,
\]
\item[(d)] $[0, T]\ni t\rightarrow f(t, Y_{t})\in L^{1}(0, T)$ and
\[
Y_{t}=\xi+\int_{t}^{T}f(r, Y_{r})\,dr
+\int_{t}^{T}dV_r-\int_{t}^{T}dA_{r}- \int_{t}^{T}dM_{r}, \quad
t\in[0,T].
\]
\end{enumerate}
\end{df}

From now on we adopt the convention that RBSDE$(\cdot,\cdot,L)$
denotes equation with lower barrier, while  RBSDE$(\cdot,\cdot,U)$
denotes equation with upper barrier.

\begin{df}
We say that a triple of processes $(Y, M, R)$ is a solution of the
reflected BSDE with terminal condition $\xi$, right-hand side
$f+dV$, lower barrier $L$ and upper barrier $U$ (RBSDE$(\xi, f+dV,
L, U)$ for short) if
\begin{enumerate}
\item[(a)] $Y$ is an $\FF$-adapted c\`adl\`ag process and of class
(D), $M\in\mathcal{M}_{0, loc}$, $R\in{ }^p\mathcal{V}_0$,
\item[(b)]$L_t\le Y_t\le U_t$ for a.e. $t\in[0, T]$,
\item[c)]for every c\`adl\`ag processes $\hat{L}, \hat{U}$ such that
$L_t\le\hat{L}_t\le Y_t\le\hat{U}_t\le U_t$ for a.e. $t\in[0, T]$
we have
\[
\int_0^T(Y_{t-}-L_{t-})\,dR_t^+=\int_0^T(U_{t-}-Y_{t-})\,dR_t^-=0,
\]
\item[(d)]$[0,T]\ni t\mapsto f(t, Y_t)\in L^1(0, T)$ and
\[
Y_t=\xi+\int_t^Tf(r, Y_r)\,dr+\int_t^TdV_r+\int_t^TdR_r
-\int_t^TdM_r, \quad t\in[0,T].
\]
\end{enumerate}
\end{df}

\begin{stw}
\label{stw2.1} Let $(Y^{i}, M^{i}, R^{i})$ be a solution of
$\mbox{\rm{RBSDE}}(\xi^{i},f^{i}+dV^{i}, L^{i}, U^i)$, $i=1, 2$.
Assume that $\xi^{1}\le\xi^{2}$, $L_{t}^{1}\le L_{t}^{2}$,
$U^1_t\le U_t^2$ for a.e. $t\in[0, T]$, $dV^{1}\le dV^{2}$, and
either $f^{1}$ satisfies $\mbox{\rm{(H1)}}$ and $f^{1}(t,
Y^{2}_{t})\le f^{2}(t, Y^{2}_{t})$ for a.e. $t\in[0, T]$ or
$f^{2}$ satisfies $\mbox{\rm{(H1)}}$ and $f^{1}(t, Y^{1}_{t})\le
f^{2}(t, Y^{1}_{t})$ for a.e. $t\in[0, T]$. Then $Y_{t}^{1}\le
Y_{t}^{2}$, $t\in[0, T]$.
\end{stw}
\begin{dow}
The proof is analogous to the proof of Proposition \ref{stw1.1}.
\end{dow}

\begin{wn}
Let assumption  $\mbox{\rm{(H1)}}$ hold. Then there exists at most
one solution of $\mbox{\rm{RBSDE}}(\xi,f+dV,L,U)$.
\end{wn}
Let us consider the following hypothesis.
\begin{enumerate}
\item[(H6)] There exists $X\in \mathcal{V}^1\oplus\MM_{loc}$
such that
\[
L_t\le X_t\le U_t \quad\mbox{\rm{for a.e. }}\,t\in[0, T], \quad
f(\cdot,X)\in L^1(\FF).
\]
\end{enumerate}

In what follows $(Y^{n, m}, M^{n, m})$ is a solution of the BSDE
\begin{align}
\label{eq2.1} \nonumber Y^{n, m}_t&=\xi+\int_t^Tf(r,Y^{n,m}_r)\,dr
+\int_t^Tn(Y^{n, m}_r-L_r)^-\,dr\\
&\quad-\int_t^Tm(Y^{n,m}_r-U_r)^+\,dr +\int_t^TdV_r-\int_t^TdM^{n,
m}_r, \quad t\in[0,T]
\end{align}
and $(\bar{Y}^n, \bar{M}^n, \bar{A}^n)$ is a solution of the RBSDE
\begin{align}
\label{eq2.2} \nonumber
\bar{Y}^{n}_t&=\xi+\int_t^Tf(r,\bar{Y}^{n}_r)\,dr
+\int_t^Tn(\bar{Y}^{n}_r-L_r)^-\,dr\\
&\quad-\int_t^T\,d\bar{A}_r^n
+\int_t^TdV_r-\int_t^Td\bar{M}^{n}_r, \quad t\in[0,T]
\end{align}
with upper barrier $U$. Write
\[
\bar{K}_t^n=\int_0^tn(\bar{Y}^n_t-L_r)^-\,dr,\quad t\in [0,T].
\]

\begin{tw}
\label{tw2.1} Assume $\mbox{\rm{(H1)--(H4)}}$. Then there exists a
unique solution $(Y, M, R)$ of $\mbox{\rm{RBSDE}}(\xi, f+dV, L,
U)$ such that $R\in\mathcal{V}_0^1$ iff $\mbox{\rm{(H6)}}$ is
satisfied. Moreover, $\bar{Y}_t^n\nearrow Y_t$, $t\in[0,T]$,
$d\bar{A}^n\le d\bar{A}^{n+1}$, $n\in\BN$, $\bar{A}_t^n\nearrow
R_t^-$, $t\in[0,T]$, $\bar{Y}^n\rightarrow Y$,
$\int_s^{\cdot}d\bar{K}^n_r\rightarrow \int_s^{\cdot} dR^+_r$ in
ucp on $[s,\tau_s)$ for every $s\in [0,T)$, where
$\tau_s=\inf\{t>s; \Delta R^+_t>0\}$, and $Y^{n, m}_t\rightarrow
Y_t$, $t\in[0, T]$, as $n,m\rightarrow\infty$.
\end{tw}
\begin{dow}
Assume that there exists a solution $(Y,M,R)$ of
RBSDE$(\xi,f+dV,L,U)$ such that $R\in\mathcal{V}_0^1$. Then
$f(\cdot,Y)\in L^1(\FF)$ by \cite[Lemma 2.3]{KR:JFA}, so (H6) is
satisfied with $X=Y$. Now assume (H6). By Theorem \ref{tw1.1}
there exists a unique solution of (\ref{eq2.2}) such that
$\bar{A}^n\in\mathcal{V}_0^1$. Moreover, by \cite[Theorem
2.7]{KR:JFA}, there exists a unique solution of (\ref{eq2.1}), and
by Theorem \ref{tw1.1},
\begin{equation}
\label{eq2.3} Y_t^{n, m}\searrow \bar{Y}_t^n, \quad t\in[0,T],
\quad EA_T^{n, m}\rightarrow E\bar{A}_T^n,
\end{equation}
where
\[
A^{n,m}_t=\int_0^tm(Y^{n, m}_r-U_r)^+\,dr,\quad t\in[0, T].
\]
By \cite[Proposition 2.1]{KR:JFA},
\begin{equation}
\label{eq3.bb}
\hat{Y}_t\le Y_t^{n, m}\le \check{Y}_t, \quad t\in[0, T],
\end{equation}
where $(\hat{Y},\hat{M},\hat{K})$ is a solution of
RBSDE$(\xi,f+dV,U)$ and  $(\check{Y},\check{M},\check{A})$ is a
solution of RBSDE$(\xi, f+dV, L)$. By (H6) there exist
$C\in\mathcal{V}_0^1$ and $N\in\MM_{0, loc}$ such that
\[
X_t=X_0-C_t+N_t, \quad t\in[0,T].
\]
Since $L_t\le X_t\le U_t$ for a.e. $t\in[0,T]$, we have
\begin{align*}
X_t&=X_T+\int_t^Tf(r, X_r)\,dr+ \int_t^Tf^-(r,
X_r)\,dr-\int_t^Tf^+(r, X_r)\,dr\\
&\quad+ \int_t^TdC_r+\int_t^TdV_r-\int_t^Tm(X_r-U_r)^+\,dr
-\int_t^TdN_r, \quad t\in[0, T].
\end{align*}
By \cite[Theorem 2.7]{KR:JFA} there exist a solution
$(\bar{X}^m,\bar{N}^m)$ of BSDE
\begin{align*}
\bar{X}_t^m&=X_T\vee\xi+\int_t^Tf(r, \bar{X}^m_r)\,dr
+\int_t^Tf^-(r, X_r)\,dr+\int_t^TdC_r\\
&\quad+\int_t^TdV_r-\int_t^Tm(\bar{X}^m_r-U_r)^+\,dr
-\int_t^Td\bar{N}^m_r, \quad t\in[0, T].
\end{align*}
By \cite[Proposition 2.1]{KR:JFA}, $\bar{X}^m_t\ge X_t$ for
$t\in[0,T]$, so $\bar{X}^m_t\ge L_t$ for a.e. $t\in[0,T]$, which
when combined with the above equation implies that
\begin{align*}
\bar{X}_t^m&=X_T\vee\xi+\int_t^Tf(r, \bar{X}^m_r)\,dr +
\int_t^Tf^-(r, X_r)\,dr+\int_t^TdC_r+\int_t^TdV_r\\
&\quad+\int_t^Tn(\bar{X}_t^m-L_r)^-\,dr
-\int_t^Tm(\bar{X}^m_r-U_r)^+\,dr- \int_t^Td\bar{N}^m_r, \quad
t\in[0,T].
\end{align*}
Therefore applying once  again \cite[Proposition 2.1]{KR:JFA} we
get
\begin{equation}
\label{eq2.4}
Y^{n, m}_t\le \bar{X}_t^m, \quad t\in[0, T].
\end{equation}
Write
\[
\bar{D}_t^m=\int_0^tm(\bar{X}_r^m-U_r)^+\,dr, \quad t\in[0,T].
\]
then by (\ref{eq2.4}),
\begin{equation}
\label{eqA.A} dA^{n, m}\le d\bar{D}^m \quad \mbox{on }[0,T],
\end{equation}
so by Theorem \ref{tw1.1},
\begin{equation}
\label{eq2.5}
\bar{X}_t^m\searrow\bar{X}_t, \quad t\in[0, T],\quad
E\bar{D}^m_T\rightarrow E\bar{D}_T,
\end{equation}
where $(\bar{X}, \bar{N}, \bar{D})$ is a solution of
RBSDE$(X_T\vee\xi, f+f^-_X+dC^++dV, U)$. From (\ref{eq2.5}) and
(\ref{eq2.3}) we conclude that
\begin{equation}
\label{eq2.6}
E\bar{A}^n_T\le E\bar{D}_T, \quad n\ge 1.
\end{equation}
By Proposition \ref{stw1.2}, for every $n\ge 1$,
\begin{equation*}
d\bar{A}^n\le d\bar{A}^{n+1}\quad\mbox{on }\,[0, T].
\end{equation*}
Therefore by (\ref{eq2.6}) and \cite[Lemma 2.2]{Peng} there exists
$\bar{A}\in{ }^p\mathcal{V}_0^1$ such that
\begin{equation}
\label{eq2.6b}
|d\bar{A}^n-d\bar{A}|_{TV}\rightarrow 0,
\end{equation}
where $|\cdot|_{TV}$ denotes the total variation norm on $[0, T]$.
Consequently,  by Theorem \ref{tw1.1},
\begin{equation}
\label{eq2.6c}
\bar{Y}_t^n\nearrow\bar{Y}_t, \quad t\in[0, T], \quad
E\bar{K}^n_T\rightarrow E\bar{K}_T,
\end{equation}
where $(\bar{Y}, \bar{M},\bar{K})$is a solution of
RBSDE$(\xi,f+dV+d\bar{A},L)$. In particular, $\bar{Y}_t\ge L_t$
for a.e. $t\in[0, T]$ and for every c\`adl\`ag process $\hat{L}$
such that $L_t\le\hat{L}_t\le\bar{Y}_t$ for a.e. $t\in[0,T]$,
\[
\int_0^T(\bar{Y}_{t-}-\hat{L}_{t-})\,d\bar{K}_t=0.
\]
On the other hand, since $(\bar{Y}^n, \bar{M}^n, \bar{A}^n)$ is a
solution of (\ref{eq2.2}), $\bar{Y}_t^n\le U_t$ for a.e. $t\in[0,
T]$ and for every c\`adl\`ag process $\hat{U}$ such that
$\bar{Y}_t^n\le\hat{U}_t\le U_t$ for a.e. $t\in[0, T]$,
\begin{equation}
\label{eq2.6a}
\int_0^T(\hat{U}_{t-}-\bar{Y}_{t-}^n)\,d\bar{A}_t^n=0.
\end{equation}
Let $\hat{U}$ be a c\`adl\`ag process such that
$\bar{Y}_t\le\hat{U}_t\le U_t$ for a.e. $t\in[0, T]$ (let us
recall that $\bar{Y}^n_t\le\bar{Y}_t$, $t\in[0, T]$). Let
$\sigma_s=\inf\{t>s; \Delta \bar{A}_t>0\}$ and let $\{S^s_p,\,p\ge
1\}$ be an announcing sequence for $\sigma_{s}$. Then
$\Delta\bar{A}_{\sigma_s}>0$ on $\{\sigma_s<\infty\}$. We may
assume that this inequality holds for every $\omega\in
\{\sigma_s<\infty\}$,  (\ref{eq2.6b}) holds for every
$\omega\in\Omega$ and that (\ref{eq2.6a}) holds for every
$\omega\in\Omega$ and $n\in\BN$. Therefore, thanks to
(\ref{eq2.6b}), for every $\omega\in\{\sigma_s<\infty\}$ there
exists $n_0(\omega)\in\BN$ such that
$(\Delta\bar{A}^n_{\sigma_s})(\omega)>0$ for $n\ge n_0(\omega)$.
From this and (\ref{eq2.6a}) we conclude that
$\bar{Y}^n_{\sigma_s-}(\omega)=\hat{U}_{\sigma_s-}(\omega)$, $n\ge
n_0(\omega)$, on $\{\sigma_s<\infty\}$. Since
$\bar{Y}_t^n\le\hat{U}_t$ for $t\in[0, T]$ and $\{\bar{Y}^n\}$ is
increasing,
$\bar{Y}^n_{\sigma_s-}(\omega)=\bar{Y}_{\sigma_s-}(\omega)
=\hat{U}_{\sigma_s-}(\omega)$ for $n\ge n_0(\omega)$ on
$\{\sigma_s<\infty\}$. Therefore for every $s\in [0,T)$ we have
\begin{equation}
\label{eq2.7}
(\hat{U}_{\sigma_s-}-Y_{\sigma_s-})\Delta\bar{A}_{\sigma_s}=0
\quad\mbox{on }\{\sigma_s<\infty\}.
\end{equation}
By (\ref{eq2.6b})--(\ref{eq2.6a}),
\[
0=\int_{s}^{s\vee S_p^s}
(\hat{U}_{t-}-\bar{Y}_{t-}^n)\,d\bar{A}_t^n
=\int_{s}^{s\vee S_p^s}
(\hat{U}_t-\bar{Y}_t^n)\,d\bar{A}_t^n\rightarrow
\int_{s}^{s\vee S_p^s}(\hat{U}_t-\bar{Y}_t)\,d\bar{A}_t.
\]
From this and (\ref{eq2.7}) we get
\begin{equation}
\label{eq2.8}
\int_0^T(\hat{U}_{t-}-\bar{Y}_{t-})\,d\bar{A}_t=0.
\end{equation}
Thus the triple $(\bar{Y},\bar{M}, \bar{R})$, where
$\bar{R}=\bar{K}-\bar{A}$, is a solution of RBSDE$(\xi, f+dV, L,
U)$. Now instead  of (\ref{eq2.2}) let us consider a solution
$(\underline{Y}^m,\underline{M}^m,\underline{K}^m)$ of the RBSDE
\begin{align*}
\underline{Y}^m_t&=\xi+\int_t^Tf(r, \underline{X}^m_r)\,dr
+\int_t^Td\underline{K}_r^m\\
&\quad-\int_t^Tm(\underline{Y}_r^m-U_r)^+\,dr
+\int_t^TdV_r-\int_t^Td\underline{M}^m_r, \quad t\in[0,T]
\end{align*}
with lower barrier $L$.
Repeating, with some obvious changes, the proofs of the first
parts of (\ref{eq2.3}), (\ref{eq2.6c}) we show that
\begin{equation}
\label{eq2.9}
\underline{Y}_t^m\nearrow\underline{Y}_t, \quad
Y^{n, m}_t\searrow\underline{Y}_t^m, \quad t\in[0, T],
\end{equation}
where $(\underline{Y}, \underline{M}, \underline{R})$ is a
solution of RBSDE$(\xi,f+dV, L,U)$ constructed analogously to
$(\bar{Y},\bar{M},\bar{R})$. By Proposition \ref{stw2.1},
$(\bar{Y},\bar{M},\bar{R})=(\underline{Y},
\underline{M},\underline{R})\equiv(Y,M,R)$, so that by
(\ref{eq2.3}) and (\ref{eq2.9}),
\[
\bar{Y}_t^n\le Y^{n, m}_t\le\underline{Y}^m_t, \quad t\in[0,T]
\]
and $Y^{n, m}_t\rightarrow Y_t$, $t\in[0, T]$, as $n,
m\rightarrow\infty$. We now show that $\bar{A}=R^-$,
$\bar{K}=R^+$. Observe that the triple $(Y, M, R^-)$ is a solution
of RBSDE$(\xi, f+dV+dR^+, U)$. Therefore by Theorem \ref{tw1.1},
\begin{equation}
\label{eq2.10} \tilde{Y}_t^m\searrow Y_t, \quad t\in[0, T], \quad
\sup_{t\in[s, s\vee S_p^s]}
\Big|\int_{s}^td\tilde{A}^m_r-\int_{s}^tdR^-_r\Big|
\rightarrow_{P}0,
\end{equation}
where the pair $(\tilde{Y}^n, \tilde{M}^n)$ is a solution of the
BSDE
\[
\tilde{Y}^m_t=\xi+\int_t^Tf(r, \tilde{Y}^m_r)\,dr
+\int_t^TdR^+_r-\int_t^Td\tilde{A}^m_r+\int_t^TdV_r-
\int_t^Td\tilde{M}^m_r, \quad t\in[0, T]
\]
and $\tilde{A}^m_t=\int_0^tm(\tilde{Y}^m_r-U_r)^+\,dr$,
$t\in[0,T]$. By (\ref{eq2.10}), $\tilde{Y}^m_t\ge L_t$ for a.e.
$t\in[0, T]$. Therefore
\begin{align*}
\tilde{Y}^m_t&=\xi+\int_t^Tf(r, \tilde{Y}^m_r)\,dr
+\int_t^Tn(\tilde{Y}_r^m-L_r)^-\,dr\\
&\quad +\int_t^TdR^+_r-\int_t^Td\tilde{A}^m_r+\int_t^TdV_r
-\int_t^Td\tilde{M}^m_r, \quad t\in[0,T]
\end{align*}
and hence,  by \cite[Proposition 2.1]{KR:JFA}, $\tilde{Y}^m_t\ge
Y^{n, m}_t$, $t\in[0,T]$. Consequently,
\begin{equation}
\label{eq2.11} dA^{n, m}\le d\tilde{A}^m\quad\mbox{on }\,[0,T]
\end{equation}
for $n,m\in\BN$. Since $d\bar{A}^m\le d\bar{A}$, Theorem
\ref{tw1.1} implies that
\[
\sup_{t\in[s, s\vee S_p^s]}
\left|\int_{s}^tdA^{n,m}_r-\int_{s}^td\bar{A}^n_r\right|
\rightarrow_{P}0
\]
as $m\rightarrow\infty$. This when combined with (\ref{eq2.6b})
and (\ref{eq2.10}), (\ref{eq2.11}) shows that for every $s\in[0,T)$,
\begin{equation}
\label{eq2.12}
d\bar{A}\le dR^- \quad\mbox{on }\,(s, \sigma_{s}).
\end{equation}
In the reasoning preceding (\ref{eq2.7}) we have showed that
$\bar{Y}^n_{\sigma_k-}(\omega)=\hat{U}_{\sigma_k-}(\omega)$  for
$n\ge n_0(\omega)$ on $ \{\sigma_s<\infty\}$. This implies that
$^p\bar{Y}_{\sigma_s}^n-\bar{Y}_{\sigma_s-}^n
\rightarrow{}^pY_{\sigma_s}-Y_{\sigma_s-}$ on $
\{\sigma_s<\infty\}$\,. Also
$\Delta\bar{A}^n_{\sigma_s}\rightarrow\Delta \bar{A}_{\sigma_s}$
on $ \{\sigma_s<\infty\}$ by (\ref{eq2.6b}). Therefore $\Delta
\bar{K}_{\sigma_s}=0$ on $ \{\sigma_s<\infty\}$. Since $dR^+\le
d\bar{K}$ by the minimality of the Jordan decomposition of $dR$,
we have $\Delta R^+_{\sigma_s}=0$ on $ \{\sigma_s<\infty\}$. Hence
$\Delta R^-_{\sigma_s}=\Delta\bar{A}_{\sigma_s}$  on $
\{\sigma_s<\infty\}$. This and (\ref{eq2.12}) show that
\[
d\bar{A}\le dR^-.
\]
In much the same way one can show that $d\bar{K}\le dR^+$.
Therefore by the minimality of the Jordan decomposition of the
measure $dR$, $dR^-=d\bar{A}$ and $dR^+=d\bar{K}$.
\end{dow}
\medskip

Let $L, U$ be c\`adl\`ag and let $(Y, M, R)$ be a solution of
RBSDE$(\xi,f+dV,L,U)$ such that $R\in\mathcal{V}_0^1$. For
$t\in[0,T]$ set
\[
\TT_t=\{\tau\in\TT; \quad t\le\tau\le T\}
\]
and consider the payoff function
\begin{equation}
\label{eq2.13} R_t(\sigma, \tau)=\int_t^{\sigma\wedge\tau}f(r,
Y_r)\,dr+
\int_t^{\sigma\wedge\tau}dV_r+L_\tau\mathbf{1}_{\{\tau<T,
\tau\le\sigma\}} +U_\sigma\mathbf{1}_{\{\sigma<\tau\}}
+\xi\mathbf{1}_{\{\sigma\wedge\tau=T\}}.
\end{equation}
The lower $\underline{W}$ and upper $\overline{W}$ values of the
stochastic game corresponding to $R_t(\cdot, \cdot)$ are defined
as
\[
\underline{W}_t=\esssup_{\tau\in \TT_t}\essinf_{\sigma\in\TT_t}
E(R_t(\sigma, \tau)|\FF_t) ,\quad
\overline{W}_t=\esssup_{\sigma\in \TT_t}\essinf_{\tau\in\TT_t}
E(R_t(\sigma, \tau)|\FF_t).
\]
The game is said to have a value if
$\underline{W}_t=\overline{W}_t$, $t\in[0,T]$.

\begin{stw}
\label{stw2.2} Let $(Y, M, R)$ be a solution of
$\mbox{\rm{RBSDE}}(\xi, f+dV, L, U)$. Then the stochastic game
associated with payoff \mbox{\rm(\ref{eq2.13})} has the value
equal to $Y$, i.e.
\[
Y_t=\overline{W}_t=\underline{W}_t, \quad t\in[0, T].
\]
\end{stw}

\begin{dow}
It is enough to repeat step by step the proof of \cite[Proposition
3.1]{Lepeltier}.
\end{dow}

\nsubsection{Nonintegrable solutions of reflected BSDEs}
\label{sec4}

In Sections \ref{sec2} and \ref{sec3} under the assumption that
(H1)--(H4) are satisfied necessary and sufficient conditions for
the existence of a solution of reflected BSDE with one and two
barriers are formulated. In the case of one barrier the necessary
and sufficient condition (H5) relates the growth of the barrier
$L$ to the generator $f$. In the case of two barriers the
corresponding condition (H6) consists of two parts. The first one,
as in the case of one barrier, relates the growth of the lower
barrier $L$ and  upper barrier $U$ to $f$. The second one, known
as Mokobodski's condition, amounts to saying that there is some
semimartingale between $L$ and $U$. The question arises whether
the solution still exists if we get rid of the conditions relating
the growth of the barriers to $f$ and we only impose minimal
integrability conditions on $L,U$ ensuring Snell envelope
representation of a possible solution, i.e. ensuring that if a
solution exists, it is of class (D). In the case of Brownian
filtration and continuous barriers the question was investigated
in \cite{Kl:EJP}. It appears that the answer is positive but in
general the reflecting process may be nonintegrable for every
$q>0$.

\begin{tw}
\label{tw4.1} Assume that \mbox{\rm(H1)--(H4)} are satisfied and
$L$ is of class $\mbox{\rm{(D)}}$. Then there exists a unique
solution $(Y, M, K)$ of \mbox{\rm RBSDE}$(\xi,f+dV,L)$. Moreover,
$Y^n_t\nearrow Y_t$, $t\in[0, T]$, $Y^n\rightarrow Y$,
$\int_{s}^{\cdot}dK^n_r\rightarrow \int_{s}^{\cdot}dK_r$ in ucp on
$[s,\tau_{s})$ for every $s\in [0,T)$, where $\tau_s=\inf\{t>s;
\Delta K_t>0\}$.
\end{tw}
\begin{dow}
Let  $(\tilde{Y}^n, \tilde{M}^n)$ be a solution of the BSDE
\[
\tilde{Y}^n_t=\xi+\int_t^Tf^+(r, \tilde{Y}^n_r)\,dr+
\int_t^Tn(\tilde{Y}^n_r-L_r)^-\,dr+\int_t^TdV_r-\int_t^Td\tilde{M}^n_r,
\quad t\in[0, T].
\]
Let $X$ be a supermartingale of class (D) majorizing $L$ (for
instance we may take the solution of RBSDE$(\xi,0,L)$ as $X$).
Then the data $(\xi,f^+,V,L)$ satisfy (H1)--(H4) and (H5) with $X$
chosen above. Therefore by Theorem \ref{tw1.1},
\[
\tilde{Y}^n_t\nearrow\tilde{Y}_t, \quad t\in[0,T],
\]
where $(\tilde{Y}, \tilde{M}, \tilde{K})$ is a solution  of
RBSDE$(\xi, f^++dV, L)$ such that $\tilde{K}\in\VV_0^1$. By
\cite[Proposition 2.1]{KR:JFA}, $\{Y^n\}$ is nondecreasing and
\begin{equation}
\label{eq4.1}
Y^0_t\le Y_t^n\le\tilde{Y}_t^n\le\tilde{Y}_t, \quad t\in[0, T].
\end{equation}
Put
\[
Y_t=\sup_{n\ge 1}Y^n_t, \quad t\in[0,T]
\]
and $\xi^k_n=Y_{\delta_k}^n$, $\xi^k=Y_{\delta_k}$, where
\[
\delta_k=\inf\{t\ge 0;\, \int_0^tf^-(r, \tilde{Y}_r)\,dr\ge k\}
\wedge T.
\]
By (\ref{eq4.1}), $Y$ is of class (D), so $\xi^k_n\nearrow\xi^k$
a.s and in $L^1$ for each $k\ge 1$. Observe that the data
$(\xi^k_n, f+dV, L)$ satisfy  (H1)--(H6) on $[0,\delta_k]$ with
$X=\tilde{Y}$. Let $(Y^{n, (k)}, M^{n, (k)})$ be a solution of
(\ref{eq1.7}) on $[0, \delta_k]$ with terminal condition
$\xi^k_n$. It is clear that $(Y^{n, (k)}_t, M^{n, (k)}_t)=(Y^n_t,
M^n_t)$, $t\in[0, \delta_k]$. By Theorem \ref{tw1.1}, $(Y^{n,
(k)}, M^{n, (k)}, K^{n, (k)})\rightarrow(Y^{(k)}, M^{(k)},
K^{(k)})$ on $[0, \delta_k]$, where $(Y^{(k)}, M^{(k)}, K^{(k)})$
is a solution of RBSDE$(\xi^k, f+dV, L)$ on $[0, \delta_k]$ and
$K^{n, (k)}_t=\int_0^tn(Y^{n, (k)}_r-L_r)^-\,dr$,
$t\in[0,\delta_k]$. Of course, $Y^{(k)}_t=Y_t$, $t\in[0,
\delta_k]$. Observe that $\{\delta_k\}$ is stationary, i.e.
\[
P(\liminf_{k\rightarrow\infty}\{\delta_k=T\})=1.
\]
Also observe that from the fact that $Y^{(k)}_t=Y_t^{(k+1)}$ for
$t\in[0,\delta_k]$ and the uniqueness of the Doob-Meyer
decomposition it follows that
\[
(Y^{(k+1)}_t, M^{(k+1)}_t, K^{(k+1)}_t) =(Y^{(k)}_t, M^{(k)}_t,
K^{(k)}_t), \quad t\in[0,\delta_k]
\]
for every $k\ge 1$. Therefore we may define process $M, K$ on
$[0,T]$ by putting
\[
M_t=M_t^{(k)}, \quad K_t=K_t^{(k)}, \quad t\in[0,\delta_k].
\]
It is clear that the triple $(Y,M,K)$ is a solution of
RBSDE$(\xi,f+dV,K)$.
\end{dow}
\medskip

The following hypothesis called the Mokobodski condition in the
literature.
\begin{enumerate}
\item[(H7)] There exists a process $X\in\VV^1\oplus\MM_{loc}$
such that
\[
L_t\le X_t\le U_t\quad\mbox{for a.e. }\,t\in[0, T].
\]
\end{enumerate}

\begin{tw}
\label{tw4.2} Assume $\mbox{\rm{(H1)--(H4), (H7)}}$. Then there
exists a unique solution $(Y, M, R)$ of
$\mbox{\rm{RBSDE}}(\xi,f+dV,L,U)$. Moreover, $\bar{Y}_t^n\nearrow
Y_t$, $Y_t^{n, m}\rightarrow Y_t$, $t\in[0, T]$,
$\bar{A}_t^n\nearrow R^{-}_t$, $\bar{Y}^n\rightarrow Y$,
$\int_s^{\cdot}d\bar{K}^n_r\rightarrow \int_s^{\cdot} dR^+_r$ in
ucp on $[s,\tau_s)$ for every $s\in [0,T)$, where
$\tau_s=\inf\{t>s; \Delta R^+_t>0\}$, and $Y^{n, m}_t\rightarrow
Y_t$, $t\in[0, T]$, as $n,m\rightarrow\infty$.
\end{tw}
\begin{dow}
The existence of a solution $(\bar{Y}^n, \bar{M}^n, \bar{A}^n)$ of
(\ref{eq2.2}) follows from Theorem \ref{tw4.1}. Let
\[
\delta_k=\inf\{t\ge 0; \,\int_0^t|f(r, X_r)|\,dr\ge k\}.
\]
Repeating the proof of Theorem \ref{tw2.1} one can first prove the
existence of  solutions of RBSDE$(\xi,f+dV,L,U)$ on the intervals
$[0,\delta_k]$. Then using these solutions and the fact that
$\{\delta_k\}$ is stationary one can construct a solution on the
whole interval $[0,T]$ (see the proof of Theorem \ref{tw4.1} for
details).
\end{dow}

\begin{stw}
Let $(Y, M, R)$ be a solution of \mbox{\rm RBSDE}$(\xi,f+dV,L,U)$.
Then for any c\`adl\`ag processes $\hat{L}, \hat{U}$ such that
$L_t\le\hat{L}_t\le Y_t\le\hat{U}_t\le U_t$ for a.e. $t\in[0, T]$,
\[
\Delta
R_{\tau_s}^+=(^pY_{\tau_s}+{}^pV_{\tau_s}-(\bar{L}_{\tau_s-}+
V_{\tau_s-}))^-\quad\mbox{on} \quad  \{\tau_s<\infty\},
\]
\[
\Delta R_{\sigma_s}^-
=(^pY_{\sigma_s}+{}^pV_{\sigma_s}-(\bar{U}_{\sigma_s-}
+V_{\sigma_s-}))^+\quad\mbox{on}\quad  \{\sigma_s<\infty\}
\]
for every $s\in [0,T)$, where
\[
\tau_s=\inf\{t>s; \Delta R^+_t>0\},\quad \sigma_s=\inf\{t>s;\Delta
R^{-}_t>0\}.
\]
\end{stw}
\begin{proof}
Follows directly from the definition of a solution of
RBSDE$(\xi,f+dV,L,U)$.
\end{proof}

\begin{wn}
Let $(Y,M,R)$ be a solution of \mbox{\rm RBSDE}$(\xi,f+dV,L,U)$.
Assume that $L_T\le\xi\le U_T$, $U,L$ are c\`adl\`ag, $\FF$ is
quasi-left continuous and the jumps of $L,U V$ are totally
inaccessible. Then $R$ is continuous.
\end{wn}

\begin{tw}
Let $(Y,M,R)$ be a solution of RBSDE$(\xi,f+dV,L,U)$ with $L$ of
the form
\begin{equation}
\label{eq4.ls}
L_t=L_0-\int_0^t\, dA_r+\int_0^t\, dN_r,\quad t\in [0,T]
\end{equation}
for some $A\in\mathcal{V}_0,\, N\in\mathcal{M}_{0,loc}$. Then
\[
dR^{+}_t\le\mathbf{1}_{\{Y_{t-}=L_{t-}\}}(f(t,L_t)\,dt+dV^{p}_t-dA^{p}_t)^{+},
\]
where $V^p,A^p$ are dual predictable projections of $V$ and $A$,
respectively.
\end{tw}
\begin{dow}
By the Tanaka-Meyer formula (see \cite[Theorem IV.70]{Protter}),
\begin{align*}
(Y_t-L_t)^+&=(Y_0-L_0)^+ -\int_0^t\mathbf{1}_{\{Y_{r-}>L_{r-}\}}f(r,Y_r)\,dr
-\int_0^t \mathbf{1}_{\{Y_{r-}>L_{r-}\}} d(V_r-A_r-R^{-}_r)\\
&\quad -\int_0^t\mathbf{1}_{\{Y_{r-}>L_{r-}\}} dR^{+}_r
-\frac12 L^0_t(S)+J_t+\int_0^t\mathbf{1}_{\{Y_{r-}>L_{r-}\}} d(M_r-N_r),
\end{align*}
where
\[
J_{t}^{+}=\sum_{0<s\le t}(\varphi(S_{s})-\varphi(S_{s-})
-\varphi'(S_{s-})\Delta S_{s}),\quad
S_{t}=Y_t-L_t,\quad \varphi(x)=x^{+}
\]
and $\varphi'$ is the left derivative of $\varphi$. Observe that
$J$  is an increasing process. By the definition of solution of
RBSDE, $S_t\ge 0$ for $t\in [0,T]$. Therefore we conclude from the
preceding equation and (\ref{eq4.ls}) that
\begin{align*}
&\int_0^t\mathbf{1}_{\{Y_{r-}=L_{r-}\}}f(r,Y_r)\,dr +\int_0^t
\mathbf{1}_{\{Y_{r-}=L_{r-}\}} d(V_r-A_r-R^{-}_r)\\
&\qquad+\int_0^t\mathbf{1}_{\{Y_{r-}=L_{r-}\}} dR^{+}_r+\frac12
L^0_t(S)+J_t-\int_0^t\mathbf{1}_{\{Y_{r-}=L_{r-}\}}\,d(M_r-N_r)=0.
\end{align*}
By the definition of a solution of RBSDE,  $\int_0^t dR^+_r
=\int_0^t\mathbf{1}_{\{Y_{r-}=L_{r-}\}} dR^{+}_r $. Hence
\begin{align*}
\int_0^t dR^{+}_r+\frac12 L^0_t(S)+J_t^p
&=-\int_0^t\mathbf{1}_{\{Y_{r-}=L_{r-}\}}f(r,Y_r)\,dr\\
&\quad+\int_0^t \mathbf{1}_{\{Y_{r-}=L_{r-}\}}
d(R^{-}_r-V^p_r+A^p_r),
\end{align*}
which leads to the desired estimate, because $dR^{+}, dR^{-}$ are
orthogonal.
\end{dow}

\nsubsection{BSDEs with two reflecting barriers: the case of
$p\in(1,2]$} \label{sec5}

In this section we show some integrability  properties of
solutions of reflected BSDEs under the assumption that the data
are in $L^p$ with $p\in(1,2]$. Except for Proposition \ref{stw5.1}
and Lemma \ref{lm5.1} we always assume that the underlying
filtration is quasi-left continuous.
\begin{stw}
\label{stw5.1} Assume that $M\in\MM_{0,loc}$, $K\in\VV_0$, $X_0$
is $\FF_0$ measurable and
\[
X_t=X_0+\int_0^tdK_r+\int_0^tdM_r,\quad t\in[0,T].
\]
Then for $p\in(1,2)$,
\begin{align*}
|X_t|^p-|X_s|^p&\ge p\int_s^t|X_{r-}|^{p-1}\hat{X}_{r-}\,dK_r+
p\int_s^t|X_{r-}|^{p-1}\hat{X}_{r-}\,dM_r\\
&\quad +\frac 1 2p(p-1)\int_s^t\mathbf{1}_{\{X_r\neq0\}}|X_r|^{p-2}\,d[X]_r^c\\
&\quad+\sum_{s<r\le t}(\Delta|X_r|^p-p|X_{r-}|^{p-1}
\hat{X}_{r-}\Delta X_r),
\end{align*}
where $\hat{x}=\frac{x}{|x|}\mathbf{1}_{\{x\neq 0\}}$,
$x\in\mathbb{R}$.
\end{stw}
\begin{dow}
Write $u_\varepsilon^p(x)=(|x|^2+\varepsilon^2)^{p/2}$, $x\in\BR$.
It is easily checked that
\[
\frac{du^p_\varepsilon}{dx}(x)=pu^{p-2}(x)x,\quad
\frac{d^2u^p_\varepsilon}{dx^2}(x)=pu^{p-2}(x)+p(p-2)u_\varepsilon^{p-4}
(x)x^2,\quad x\in\BR.
\]
By the It\^o-Meyer formula,
\begin{align}
\label{eq5.3} u^p_\varepsilon(X_t)-u_\varepsilon^p(X_s)&=\int_s^t
\frac{du^p_\varepsilon}{dx}(X_{r-})\,dX_r+\frac 1 2\int_s^t
\frac{d^2u^p_\varepsilon}{dx^2}(X_r)\,d[X]_r^c \nonumber\\
&\quad+\sum_{s<r\le t}(\Delta u_\varepsilon^p(X_r)-
\frac{du^p_\varepsilon}{dx}(X_{r-})\Delta X_r) \nonumber\\
&=\int_s^tpu_\varepsilon^{p-2}(X_{r-})
X_{r-}\,dK_r+\int_s^tpu_\varepsilon^{p-2}(X_{r-})X_{r-}\,dM_r\nonumber\\
&\quad+\frac12\int_s^t(pu_\varepsilon^{p-2}(X_r)
+p(p-2)u_\varepsilon^{p-4}(X_r)X_r^2)\,d[X]_r^c\nonumber\\
& \quad+\sum_{s<r\le t}(\Delta u_\varepsilon^p(X_r)
-pu_\varepsilon^{p-2}(X_{r-})X_{r-}\Delta X_r).
\end{align}
It is clear that
\begin{equation}
\label{eq5.2}
u_\varepsilon^p(X_t)-u_\varepsilon^p(X_s)\rightarrow|X_t|^p-|X_s|^p.
\end{equation}
Observe that $pu_\varepsilon^{p-2}(x)x\rightarrow
p|x|^{p-1}\hat{x}$, $x\in\BR$, and, by convexity of
$u_\varepsilon^p$, $\Delta
u_\varepsilon^p(X_r)-pu_\varepsilon^{p-2}(X_{r-})X_{r-}\Delta
X_r\ge 0$, $r\in[0,T]$. Therefore applying  Fatou's lemma we get
\begin{align}
\label{eq5.4} &\liminf_{\varepsilon\rightarrow 0^+}\sum_{s<r\le t}
(\Delta u_\varepsilon^p(X_r)
-pu_\varepsilon^{p-2}(X_{r-})X_{r-}\Delta X_r))\nonumber\\
&\qquad\ge \sum_{s<r\le t}
(\Delta|X_r|^p-p|X_{r-}|^{p-1}\hat{X}_{r-}\Delta X_r).
\end{align}
By the Lebesgue dominated convergence theorem,
\begin{equation}
\label{eq5.5}
\int_s^tpu_\varepsilon^{p-2}(X_{r-})X_r\,dK_r\rightarrow
\int_s^tp|X_{r-}|^{p-1}\hat{X}_{r-}\,dK_r
\end{equation}
and
\begin{equation}
\label{eq5.6}
\int_s^tpu_\varepsilon^{p-2}(X_{r-})X_r\,dK_r\rightarrow
\int_s^tp|X_{r-}|^{p-1}\hat{X}_{r-}\,dM_r.
\end{equation}
From the identity
\[
u_\varepsilon^q(x)|x|^2=u_\varepsilon^{q+2}(x)-\varepsilon^2
u_\varepsilon^q(x),\quad x,q\in\BR
\]
it follows that
\begin{align}
\label{eq5.7} &\int_s^t(pu_\varepsilon^{p-2}(X_{r})+
p(p-2)u_\varepsilon^{p-4}(X_r))\,d[X]_r^c \nonumber\\
&\qquad= \int_s^tp(p-1)u_\varepsilon^{p-4}(X_{r})|X_r|^2\,d[X]^c_r
+\int_s^tp\varepsilon^2u_\varepsilon^{p-4}(X_{r})\,d[X]_r^c.
\end{align}
We also have
\begin{align}
\label{eq5.8}
\nonumber\int_s^tu_{\varepsilon}^{p-4}(X_r)|X_r|^2\,d[X]_r^c
&=\int_s^t\left(\frac{|X_r|}
{u_\varepsilon(X_r)}\right)^{4-p}|X_r|^{p-2} \mathbf{1}_{\{X_r\neq
0\}}\,d[X]_r^c
\\&\nearrow\int_s^t\mathbf{1}_{\{X_r\neq 0\}}|X_r|^{p-2}\,d[X]_r^c.
\end{align}
From (\ref{eq5.3}) and (\ref{eq5.2})--(\ref{eq5.8}) we deduce the
the desired result.
\end{dow}
\medskip

Now we are going to prove some a priori estimates for solutions
of reflected BSDEs. For this we need the following lemma.

\begin{lm}
\label{lm5.1} Let $p\in(1,2]$ and let $\varphi(x)=|x|^p$,
$x\in\BR$. Then for every $x,y\in\BR$,
\[
\varphi(x)-\varphi(y)-\varphi'(y)(x-y)\ge \frac12
\mathbf{1}_{\{|x|\vee|y|\neq 0\}}\varphi''(|x|\vee|y|)(x-y)^2.
\]
\end{lm}
\begin{dow}
By using a mollification of $\varphi$ one can easily show that for
$x\neq y$,
\begin{equation}
\label{eq5.9} \varphi(x)-\varphi(y)-\varphi'(y)(x-y)=
\Big(\int_0^1\int_0^1\alpha\varphi''(y+\alpha\beta(x-y))\,d
\alpha\,d\beta\Big)(x-y)^2.
\end{equation}
For $\varepsilon>0$ set
$\varphi^p_\varepsilon(z)=(|z|^2+\varepsilon^2)^{p/2}$. A direct
computation shows that
\begin{align*}
\frac{d^2\varphi^p_\varepsilon}{dz^2}(z)
&=p(p-1)\varphi^{p-4}_\varepsilon(z)z^2
+\varepsilon^2\varphi^{p-4}_\varepsilon(z)\\
& \ge p(p-1)\varphi^{p-4}_\varepsilon(z)z^2
=(z^2+\varepsilon^2)^{(p-2)/2}-\varepsilon^2(z^2+\varepsilon^2)^{(p-4)/2}.
\end{align*}
Let $z\neq 0$ and $x_1\le z\le x_2$. Then
\[
\frac{d^2\varphi^p_\varepsilon}{dz^2}(z)\ge p(p-1)
((|x_1|\vee|x_2|)^2+\varepsilon^2)^{(p-2)/2}
-\varepsilon(z^2+\varepsilon^2)^{(p-4)/2}.
\]
Since $\frac{d^2\varphi^p_\varepsilon}{dz^2}(z)
\rightarrow\varphi''(z)$ for $z\neq 0$, letting
$\varepsilon\rightarrow0$ in the above inequality we get
\[
\varphi''(z)\ge p(p-1)(|x_1|\vee|x_2|)^{p-2}=\varphi''(|x_1|\vee|x_2|).
\]
From this we conclude that if $x\neq 0$ or $y\neq 0$ then
\[
\int_0^1\int_0^1\alpha\varphi''(y+\alpha\beta(x-y))\,d\alpha\,d\beta\ge
\frac 1 2\varphi''(|x|\vee|y|).
\]
This when combined with (\ref{eq5.9}) gives the desired result.
\end{dow}
\medskip

Let us consider the following hypothesis.
\begin{enumerate}
\item[(A)]There exist $\lambda\ge 0$, $\mu\in\BR$ and a non-negative
progressively measurable process $f_t$ such that for every
$y\in\BR$,
\[ \hat{y}f(t,y)\le f_t+\mu|y|,\quad
dt\otimes dP\mbox{-a.s.},
\]
where $\hat{y}=\frac{y}{|y|}\mathbf{1}_{\{y\neq0\}}$.
\end{enumerate}

In the remainder of this section we assume that $\FF$ is
quasi-left continuous.
\begin{stw}
\label{stw5.2} Assume $\mbox{\rm{(A)}}$ and that $\xi\in
L^p(\FF_T)$, $f_t\in L^p(\FF)$, $V\in\VV_0^p$, $Y\in
\mathcal{S}^p$for some $p\in(1,2]$. Moreover, assume that $Y$ is a
semimartingale and denote by $M$ its martingale part in the
Doob-Meyer decomposition. Write $(Y^\alpha,Z^\alpha)=(e^{\alpha
t}Y_t,e^{\alpha t} Z_t)$, $dV^\alpha_t=e^{\alpha t}\,dV_t$ and
\[
f^\alpha(t,y,z)=e^{\alpha t}f(t,e^{-\alpha t}y,e^{-\alpha
t}z)-\alpha y.
\]
Then if
\begin{align}
\label{eq.al}
\nonumber
&|Y^\alpha_s|^p+\frac 1 2p(p-1)\int_s^t|Y^\alpha_r|^{p-2}
\mathbf{1}_{\{Y^\alpha_r\neq 0\}}\,d[Y^\alpha]_r^c+
\sum_{s<r\le t}(\Delta|Y^\alpha_r|^p-p|Y^\alpha_{r-}|^{p-1}
\hat{Y}^\alpha_{r-}\Delta Y^\alpha_r)\\&\nonumber
\quad\le |Y^\alpha_t|^p+p\int_s^t|Y^\alpha_r|^{p-1}
\hat{Y}^\alpha_rf^\alpha(r,Y^\alpha_r,Z^\alpha_r)\,dr+
p\int_s^t|Y^\alpha_{r-}|^{p-1}\hat{Y}^\alpha_{r-}\,dV^\alpha_r\\&
\quad\quad-p\int_s^t|Y^\alpha_{r-}|^{p-1}
\hat{Y}^\alpha_{r-}\,dM_r,
\quad 0\le s\le t\le T
\end{align}
for some $\alpha\ge \mu$ then
there is $C>0$ depending only on $p$ such that
\begin{align*}
&E\sup_{t\le T}|Y^\alpha_t|^p
+E\Big(\int_0^Td[
M]_r\Big)^{p/2}\\
&\qquad\le CE\Big(|\xi^\alpha|^p+\big(\int_0^Td|V^\alpha|_r\big)^p
+\big(\int_0^Tf^\alpha_r\,dr\big)^p\Big),
\end{align*}
where $\xi^\alpha=e^{\alpha T}\xi,\, f^{\alpha}_r=e^{\alpha
t}f_r$.
\end{stw}
\begin{dow}
For simplicity we assume that $\alpha=0$. We only consider the
case $p\in(1,2)$. The case $p=2$ is included in the assertion  of
Proposition \ref{stw7.2}. By assumption (A), for every
$\tau\in\mathcal{T}$ we have
\begin{align}
\label{eq5.10} \nonumber &|Y_t|^p+\frac12
p(p-1)\int_t^\tau|Y_r|^{p-2}\mathbf{1}_{\{Y_r\neq 0\}}\,d[Y]^c_r
+\sum_{t<r\le \tau}(\Delta|Y_r|^p-p|Y_{r-}|^{p-1}
\hat{Y}_{r-}\Delta Y_r)\\
&\qquad\le|Y_\tau|^p+p\int_t^\tau(|Y_r|^{p-1}f_r+\mu|Y_r|^p)\,dr
+p\int_t^\tau|Y_r|^{p-1}\,d|V|_r\nonumber\\
&\qquad\quad-p\int_t^\tau|Y_r|^{p-1}\hat{Y}_r\,dM_r.
\end{align}
By Lemma \ref{lm5.1},
\begin{align*}
&\sum_{t<r\le\tau}(\Delta|Y_r|^p-p|Y_{r-}|^{p-1}\hat{Y}_{r-}\Delta Y_r)\\
&\qquad\ge \frac12p(p-1)\sum_{t<r\le\tau}
\mathbf{1}_{\{|Y_{r}|\vee|Y_{r-}|\neq 0\}}
(|Y_r|\vee|Y_{r-}|)^{p-2}|\Delta Y_r|^2 \nonumber\\
&\qquad=\frac12p(p-1)\int_t^\tau\mathbf{1}_{\{|Y_{r}|\vee|Y_{r-}|\neq0\}}
(|Y_r|\vee|Y_{r-}|)^{p-2}|\,d[Y]_r^d.
\end{align*}
The above inequality when combined with (\ref{eq5.10})  and the
fact that $\mu\le\alpha\le0$  gives
\begin{align}
\label{eq5.12} \nonumber &|Y_t|^p+\frac12p(p-1)\int_t^\tau
\mathbf{1}_{\{|Y_{r}|\vee|Y_{r-}|\neq 0\}}
(|Y_r|\vee|Y_{r-}|)^{p-2}\, d[Y]_r \\
&\qquad\le|Y_\tau|^p+p\int_t^\tau |Y_r|^{p-1}f_r\,dr+
\int_t^\tau|Y_r|^{p-1}\,d|V|_r
-p\int_t^\tau|Y_r|^{p-1}\hat{Y}_r\,dM_r.
\end{align}
Since the filtration is quasi-left continuous, $d[Y]\ge d[M]$.
Thus
\begin{align}
\label{eq5.19} & E\int_t^\tau\mathbf{1}_{\{|Y_{r}|\vee|Y_{r-}|\neq
0\}} (|Y_r|\vee|Y_{r-}|)^{p-1}\,d[Y]_r \nonumber\\
&\qquad\ge E\int_t^\tau\mathbf{1}_{\{|Y_{r}|\vee|Y_{r-}|\neq 0\}}
(|Y_r|\vee|Y_{r-}|)^{p-1}\, d[M]_r.
\end{align}
For $k\in\BN$ write
$\tau_k=\sigma_k\wedge\delta_k$, where $\{\sigma_k\}$ is a
fundamental sequence for the local martingale
$\int_0^\cdot|Y_r|^{p-1}\hat{Y}_{r-}\,dM_r$
and
\[
\delta_k=\inf\{t\ge 0,\int_0^t(|Y_r| \vee
|Y_{r-}|)^{p-2}\mathbf{1}_{\{|Y_r|\vee|Y_{r-}|\neq 0\}}
\,d[M]_r\ge k\}.
\]
Substituting (\ref{eq5.19}) into (\ref{eq5.12}) and then replacing
$\tau$ by $\tau_k$ in (\ref{eq5.12}), integrating and  letting
$k\rightarrow\infty$  we get
\begin{equation}
\label{eq5.14}
E|Y_t|^p+\frac 1 4p(p-1)E\int_0^T\mathbf{1}_{\{|Y_r|\vee|Y_{r-}|\neq 0\}}
(|Y_r|\vee|Y_{r-}|)^{p-2}\,d[M]_r\le EX,
\end{equation}
where
$X=|\xi|^p+p\int_0^T|Y_r|^{p-1}f_r\,dr+p\int_0^T|Y_r|^{p-1}\,d|V|_r$.
Furthermore, by (\ref{eq5.12}),
\begin{align*}
E\sup_{t\le T}|Y_t|^p&\le EX+pE\sup_{t\le T}
\Big|\int_t^T
|Y_{r-}|^{p-1}\hat{Y}_{r-}\,dM_r\Big|\\
&\le EX+c_1pE\sup_{t\le T}|Y_t|^{p/2}
\Big(\int_t^T
(|Y_r|\vee|Y_{r-}|)^{p-2}
\mathbf{1}_{\{|Y_r|\vee|Y_{r-}|\neq0\}}\,d[M]_r\Big)^{1/2}\\
& \le EX+\beta E\sup_{t\le T}|Y_t|^p+\beta^{-1} c_1p
E\int_t^T(|Y_r|\vee|Y_{r-}|)^{p-2}
\mathbf{1}_{\{|Y_r|\vee|Y_{r-}|\neq0\}} \,d[M]_r.
\end{align*}
Taking $\beta>0$ sufficiently small we get
\begin{equation}
\label{eq5.15}
E\sup_{t\le T}|Y_t|^p\le c_2EX.
\end{equation}
Therefore
\begin{align}
\label{eq5.26} &E\Big(\int_0^T\,d[M]_r\Big)^{p/2}
=E\Big(\int_0^T(|Y_r|\vee|Y_{r-}|+\varepsilon)^{2-p}
(|Y_r|\vee|Y_{r-}|+\varepsilon)^{p-2}
\,d[M]_r\Big)^{p/2}\nonumber\\
&\quad \le E(\sup_{t\le T}|Y_t|^{2-p}+\varepsilon^{2-p})^{p/2}
\int_0^T(|Y_r|\vee|Y_{r-}|+\varepsilon)^{p-2}
\,d[M]_r\Big)^{p/2}\nonumber\\
&\quad\le\Big(E(\sup_{t\le T}
|Y_t|^{(2-p)}+\varepsilon^{2-p})^{(p/2)(2/p)^*}\Big)^{1/(2/p)^*}\nonumber\\
&\qquad\qquad\qquad\qquad\qquad\qquad
\times\Big(E\int_0^T(|Y_r|\vee|Y_{r-}|+\varepsilon)^{p-2}
\,d[M]_r\Big)^{p/2},
\end{align}
where $(2/p)^*$ is the H\"older conjugate to $2/p$. By
(\ref{eq5.14}) we may pass in (\ref{eq5.26}) to the limit as
$\varepsilon\rightarrow0$. We then get
\begin{align*}
E\Big(\int_0^T\,d[ M]_r\Big)^{p/2}\le(E\sup_{t\le
T}|Y_t|^p)^{(2-p)/2} \Big(E\int_0^T(|Y_r|\vee|Y_{r-}|)^{p-2}
\mathbf{1}_{\{|Y_r|\vee|Y_{r-}|\neq0\}} \,d[M]_r\Big).
\end{align*}
By Young's inequality,
\[
E\Big(\int_0^T\,d[ M]_r\Big)^{p/2}
\le\frac{2-p}{2}E\sup_{t\le T}|Y_t|^p+\frac{p}{2}
E\int_0^T(|Y_r|\vee|Y_{r-}|)^{p-2}
\mathbf{1}_{\{|Y_r|\vee|Y_{r-}|\neq0\}}\,d[M]_r.
\]
This and (\ref{eq5.14}) imply that $Z\in M^p$ and
\begin{equation}
\label{eq5.17} E\sup_{t\le T}|Y_t|^p
+E\Big(\sum_{i=1}^\infty\int_0^T|Z^i_r|^2\,d \langle
M^i\rangle_r\Big)^{p/2}\le c_3 EX.
\end{equation}
Finally, observe that
\[
EX\le\beta E\sup_{t\le T}|Y_t|^p+\beta^{-1}c_4
E\Big(|\xi|^2+\big(\int_0^Tf_r\,dr\big)^p+\big(\int_0^Td|V|_r\big)^p\Big).
\]
From this and (\ref{eq5.17}) the desired estimate follows.
\end{dow}
\medskip

For $p\in(1,2]$ we will use the following modifications to  (H4),
(H6):
\begin{enumerate}
\item[(H4${}^*$)]$\xi\in L^p(\FF_T)$, $V\in\mathcal{V}_0^p$, $f(\cdot, 0)\in
L^p(\FF)$.
\item[(H6${}^*$)] There exists $X\in \mathcal{V}^p\oplus\MM^p$
such that
\[
L_t\le X_t\le U_t \quad\mbox{\rm{for a.e. }}t\in[0,T], \quad
\int_0^{T}|f(t,X_t)|\,dt\in L^p(\FF_T).
\]
\end{enumerate}

\begin{stw}
\label{stw3.p} Assume that \mbox{\rm{(H1)--(H3)}} and
\mbox{\rm(H4${}^*$)} with $p\in(1,2]$ are satisfied. Then there
exists a solution $(Y,M)\in \mathcal{S}^p\otimes\mathcal{M}^p_0$
of BSDE$(\xi,f+dV)$. Moreover, $\int_0^T|f(t,Y_t)|\,dt\in
L^p(\FF_T)$.
\end{stw}
\begin{dow}
By \cite[Theorem 2.7]{KR:JFA} there exists a solution $(Y,M)$ of
BSDE$(\xi,f+dV)$ such that $Y$ is of class (D) and $Y\in \mathcal{S}^q$ for
$q\in(0,1)$. Moreover, $Y^n\rightarrow Y$ in $\mathcal{S}^q$, $q\in(0,1)$,
where $(Y^n,M^n)\in \mathcal{S}^2\otimes\MM^2$ is a solution of
BSDE$(\xi^n,f^n+dV^n)$ and
\[\xi^n=T_n(\xi),\quad
f_n(t,y)=f(t,y)-f(t,0)+T_n(f(t,0)),\quad
V_t^n=\int_0^t\mathbf{1}_{\{|V|_s\le n\}}\,dV_s.
\]
By Proposition \ref{stw5.2},
\[
E\sup_{t\le T}|Y_t^n|^p\le CE\Big(|\xi^n|^p
+\big(\int_0^T|f_n(r,0)|\,dr\big)^p+\big(\int_0^Td|V^n|_r\big)^p\Big).
\]
Letting $n\rightarrow\infty$ shows that $Y\in \mathcal{S}^p$ and
\[
E\sup_{t\le T}|Y_t|^p\le CE\Big(|\xi|^p +
\big(\int_0^T|f(r,0)|\,dr\big)^p+\big(\int_0^Td|V|_r\big)^p\Big).
\]
By \cite[Lemma 2.5]{KR:JFA}, $M\in\MM^p$. Hence, by \cite[Lemma
2.3]{KR:JFA},
\[
E\big(\int_0^T|f(r,Y_r)|\,dr\big)^p\le CE
\Big(|\xi|^p+\big(\int_0^T|f|(r,0)\,dr\big)^p
+\big(\int_0^T\,d|V|_r\big)^p\Big),
\]
and the proof is complete.
\end{dow}

\begin{lm}
\label{lm5.2} Assume \mbox{\rm(H1)--(H3), (H4${}^*$), (H6${}^*$)}
with $p\in(1,2]$ are satisfied. Then there exists a solution
$(Y,Z,R)$ of $\mbox{\rm{RBSDE}}(\xi,f+dV,L,U)$ such that $Y\in
\mathcal{S}^p$ and $\int_0^T|f(t,Y_t)|\,dt\in L^p(\FF_T)$.
\end{lm}
\begin{dow}
By Theorem \ref{tw2.1}  there exists a solution $(Y,Z,R)$ of
RBSDE$(\xi,f+dV,L,U)$. By Theorem \ref{tw2.1}, (\ref{eq3.bb}) and
(A1) it is enough to prove the integrability properties of $Y$
stated in the lemma in  case $(Y,Z,R)$ is a solution of RBSDE with
one reflecting barrier. So, let us assume that $(Y,Z,R)$ is the
solution of RBSDE$(\xi,f+dV,L)$. Then the desired properties of
$Y$ follow from Theorem \ref{tw1.1}, (\ref{eq1.13}) and
Proposition \ref{stw3.p}.
\end{dow}

\begin{tw}
\label{tw3.p} Assume that \mbox{\rm(H1)--(H3), (H4${}^*$)} with
$p\in(1,2]$ are satisfied. Then there exists a solution
$(Y,M,R)\in \mathcal{S}^p\otimes\mathcal{M}^p\otimes
\mathcal{V}^p_0$ of RBSDE$(\xi,f+dV,L,U)$ iff \mbox{\rm(H6${}^*$)}
is satisfied.
\end{tw}
\begin{dow}
Assume that (H1)--(H3), (H4${}^*$) are satisfied and $(Y,M,R)\in
\mathcal{S}^p\otimes\mathcal{M}_0^p\otimes\mathcal{V}^p_0$ is a
solution of RBSDE$(\xi,f+dV,L,U)$. Then by \cite[Lemma
2.5]{KR:JFA}, $\int_0^T|f(r,Y_r)|\,dr\in L^p(\FF_T)$. Therefore
(H6${}^*$) is satisfied with $X=Y$. Now assume that (H1)--(H3),
(H4${}^*$), (H6${}^*$) are satisfied. First observe that thanks to
(\ref{eqA.A}) we may assume that $(Y,M,R)$ is a solution of RBSDE
with one reflecting barrier, say lower, i.e. we may assume that
$R$ is an increasing process. By Lemma \ref{lm5.2}, $Y\in
\mathcal{S}^p$ and $E(\int_0^T|f(r,Y_r)|\,dr)^p<\infty$. From
these  properties of $Y$ and th fact that $R$ is predictable it
follows that there exists a stationary sequence $\{\tau_k\}\subset
\mathcal{T}$ such that
\[
E([M]_{\tau_k})^{p/2}+E\big(\int_0^{\tau_k}\,d
|R|_{r}\big)^{p}<\infty,\quad k\ge1.
\]
By It\^o's formula and Young's inequality,
\begin{align*}
E([M]_{\tau_k})^{p/2} &\le c_pE\Big(\sup_{t\le T}|Y_t|^p
+\big(\int_0^{\tau_k}|f(r,Y_r)|\,dr\big)^{p}\\
&\qquad\qquad+\big(\int_0^{\tau_k}d|V|_r\big)^p
+\big(\int_0^{\tau_k}|Y_{r-}|\,d|R|_r\big)^{p/2}\Big).
\end{align*}
Using once again Young's we see that for every $\alpha>0$,
\begin{align}
\label{eq.bracket} E([M]_{\tau_k})^{p/2}&\le c_pE\Big(\sup_{t\le
T}|Y_t|^p +\big(\int_0^{\tau_k}|f(r,Y_r)|\,dr\big)^{p}\nonumber\\
&\qquad\qquad+\big(\int_0^{\tau_k}d|V|_r\big)^p
+\alpha\big(\int_0^{\tau_k}\,d|R|_r\big)^{p}\Big).
\end{align}
On the other hand, since $R_t=Y_0-Y_t-\int_0^t f(r,Y_r)\,dr
-\int_0^t dV_r+\int_0^t dM_r$ for $t\in [0,T]$, applying the
Burkholder-Davis-Gundy inequality we obtain
\[
E\big(\int_0^{\tau_k}\,d |R|_{r}\big)^{p} \le
c_pE\Big(\sup_{t\le T}|Y_t|^p
+\big(\int_0^{\tau_k}|f(r,Y_r)|\,dr\big)^{p}
+\big(\int_0^{\tau_k}d|V|_r\big)^p+([M]_{\tau_k})^{p/2}\Big).
\]
The above inequality and (\ref{eq.bracket}) imply that
$M\in\mathcal{M}^p$. Hence $R\in \mathcal{V}^p_0$, because  we
already know that $Y\in \mathcal{S}^p$.
\end{dow}

\nsubsection{Reflected BSDEs with generator depending on $z$}

Assume that $\FF$ satisfies the usual conditions and the Hilbert
$L^2(\FF_T)$ is separable. Then (see \cite{DavisVaraiya,Protter})
there exists a sequence $\{M^i\}\subset\MM^2_0$ such that
$\{M_i\}$ are orthogonal, i.e. $EM_T^iM_T^j=0$ for $i\neq j$, and
for every $N\in\MM^2$,
\begin{equation}
\label{eq5.1}
N_t=N_0+\sum_{i=0}^\infty\int_0^tZ_r^i\,dM_r^i,\quad t\in[0,T]
\end{equation}
for some sequence $\{Z^i\}$ of predictable processes such that
\[
E\sum_{i=0}^\infty\int_0^T|Z^i_t|^2\,d\langle M^i\rangle_t<\infty.
\]
For given $A\in\VV^1_0$ let us denote by $\mu_A$ the measure on
$\BB([0,T])\otimes\FF_T$ defined as
\[
\mu_A(B)=E\int_0^T\mathbf{1}_B(t,w)\,dA_t, \quad
B\in\BB([0,T])\otimes\FF_T.
\]
It is known that the sequence $\{M^i\}$ may be chosen so that
$\mu_{\langle M^i\rangle}\gg\mu_{\langle M^j\rangle}$ for $i<j$.
In that case the sequence $\{M^i\}$ is unique in the following
sense: if $\{\hat{M}_i\}\subset\MM^2_0$ is an another sequence
satisfying the same conditions as $\{M^i\}$ then $\mu_{\langle
M^i\rangle}$ is equivalent to $\mu_{\langle \hat{M}^i\rangle}$ for
every $i\in\BN$. By using the localization procedure one can show
that every locally square integrable $\FF$ martingale admits
representation (\ref{eq5.1}) with $\{Z^i\}$ such that
\begin{equation}
\label{eq5.2a} P\Big(\sum_{i=1}^\infty\int_0^T|Z_t^i|^2\,d\langle
M^i\rangle_t<\infty\Big)=1.
\end{equation}

Set
\[
m^i(t,w)=\frac{d\mu^c_{\langle M^i\rangle}}{dt\otimes dP}(t,w),
\quad (r,w)\in[0,T]\times\Omega,
\]
where $\mu^c_{\langle M^i\rangle}$ is the absolutely continuous
part, with respect to $dt\otimes P$, of the measure $\mu_{\langle
M^i\rangle}$. By $M^0$ we denote the space of all processes
$Z=(Z^1,Z^2,\dots)$ such that $Z^i$ is predictable for each
$i\in\BN$ and (\ref{eq5.2a}) is satisfied. By $M^p$, $p\ge 1$, we
denote the space
\[
M^p=\{Z\in M^0;E\Big(\sum_{i=1}^\infty\int_0^T|Z_t^i|^2\, d\langle
M^i\rangle_t\Big)^{p/2}<\infty\}.
\]
We also use the following notation
\[
\|z\|^2_{M_t}=\sum_{i=1}^\infty|z^i|^2m^i(t,w),\quad
z\in\BR^{\infty}.
\]

Let $\xi$ be an $\FF_T$ measurable random variable, $V\in\VV_0$,
$L,U$ be progressively measurable processes and
$f:[0,T]\times\Omega\times\BR\times\BR^\infty\rightarrow\BR$ be
such that $f(\cdot,\cdot,y,z)$ is progressively measurable for
every $(y,z)\in\BR\times\BR^\infty$. We will need the following
hypotheses.
\begin{enumerate}
\item[(A1)] There is $\mu\in\BR$ such that for a.e. $t\in[0,T]$
and every $y,y'\in\BR$, $z\in\BR^\infty$,
\[
(f(t,y,z)-f(t,y',z))(y-y')\le\mu|y-y'|^2.
\]
\item[(A2)] $[0,T]\ni t\mapsto f(t,y,z)\in L^1(0,T)$ for every
$y\in\BR$, $z\in\BR^\infty$,
\item[(A3)] $\BR\ni y\mapsto f(t,y,z)$ is continuous for a.e. $t\in[0,T]$
and for every $z\in\BR^\infty$,
\item[(A4)] There is $\lambda\ge 0$ such that for a.e. $t\in[0,T]$
and every $y\in\BR$, $z,z'\in\BR^\infty$,
\[
|f(t,y,z)-f(t,y,z')|\le\lambda\|z-z'\|_{M_t}.
\]
\item[(A5)] $\xi\in L^2(\FF_T)$, $V\in\VV_0^2$, $f(\cdot,0,0)\in L^2(\FF)$,
\item[(A6)] There exists $X\in\VV^2\oplus\MM^2$ such that $L_t\le X_t\le U_t$
for a.e. $t\in[0,T]$ and $f(\cdot,X,0)\in L^2(\FF)$.
\item[(A${}^*$)]There exist $\lambda\ge 0$, $\mu\in\BR$ and a non-negative
progressively measurable process $f_t$ such that for every
$y\in\BR$ and $z\in\BR^\infty$,
\[
\hat{y}f(t,y,z)\le f_t+\mu|y|+\lambda\|z\|_{M_t},\quad
dt\otimes dP\mbox{-a.s.}
\]
\end{enumerate}

\begin{df}
We say that a triple $(Y,Z,R)$ is a solution of
RBSDE$(\xi,f+dV,L,U)$ if $Z\in M^0$ and the triple $(Y,M,R)$,
where $M_t=\sum_{i=1}^\infty \int_0^t Z^i_r\,dM^i_r$, $t\in[0,T]$,
is a solution of RBSDE$(\xi,\hat{f}+dV,L,U)$ with
\[\hat{f}(t,y)=f(t,y,Z_t).
\]
\end{df}

\begin{stw}
\label{stw7.2} Assume that \mbox{\rm(A${}^*$)} is satisfied and
$\xi\in L^2(\FF_T)$, $f_t\in L^2(\FF)$, $V\in\VV_0^2$, $(Y,Z)\in
\mathcal{S}^2\otimes M^0$. Moreover, assume that $Y$ is a
semimartingale and its martingale part in the Doob-Meyer
decomposition is of the form
$M=\sum_{i=1}^{\infty}\int_0^{\cdot}Z^i_r\,dM^i_r$. Write
$(Y^\alpha,Z^\alpha)=(e^{\alpha t}Y_t,e^{\alpha t} Z_t)$,
$dV^\alpha_t=e^{\alpha t}\,dV_t$ and
\[
f^\alpha(t,y,z)=e^{\alpha t}f(t,e^{-\alpha t}y,e^{-\alpha
t}z)-\alpha y.
\]
Then if
\begin{align}
\label{eq.al7}
\nonumber
|Y^\alpha_s|^2+\int_s^t\,d[Y^\alpha]_r
&\le |Y^\alpha_t|^2+2\int_s^t Y^\alpha_r
f^\alpha(r,Y^\alpha_r,Z^\alpha_r)\,dr+
2\int_s^tY^\alpha_{r-}\,dV^\alpha_r\\&
\quad\quad-2\sum_{i=1}^\infty\int_s^tY^\alpha_{r-}
Z_r^{\alpha,i}\,dM_r^i,
\quad 0\le s\le t\le T
\end{align}
for some $\alpha\ge \mu+\lambda^2$ then $Z\in M^2$ and
there is $C>0$ such that
\begin{align*}
&E\sup_{t\le T}|Y^\alpha_t|^2
+E\Big(\sum_{i=1}^\infty\int_0^T|Z_r^{\alpha,i}|^2\,d\langle
M^i\rangle_r\Big)\\
&\qquad\le CE\Big(|\xi^\alpha|^2+\big(\int_0^Td|V^\alpha|_r\big)^2
+\big(\int_0^Tf^\alpha_r\,dr\big)^2\Big),
\end{align*}
where $\xi^\alpha=e^{\alpha T}\xi,\, f^{\alpha}_r=e^{\alpha
t}f_r$.
\end{stw}
\begin{dow}
For simplicity we assume that $\alpha=0$. By assumption (A${}^*$),
for every $\tau\in\mathcal{T}$ we have
\begin{align}
\label{eq7.10} \nonumber &|Y_t|^2+\int_s^{\tau}\,d[Y^\alpha]_r
\le|Y_\tau|^2+2\int_t^\tau(|Y_r|f_r+\mu|Y_r|^2)\,dr
+2\int_t^\tau|Y_r|\,d|V|_r\\
&\qquad\quad+ 2\lambda\int_t^\tau|Y_r|\|Z_r\|_{M_r}\,dr
-2\sum_{i=1}^\infty \int_t^\tau Y_{r-} Z_r^i\,dM_r^i.
\end{align}
We have
\[
2\lambda|Y_r|\|Z_r\|_{M_r}\le
2\lambda^2|Y_r|^{2}+\frac{1}{2}\|Z_r\|_{M_r}.
\]
Since $\mu+\lambda^2\le\le\alpha\le 0$, from the above inequality
and (\ref{eq7.10}) it follows that
\begin{align}
\label{eq7.13} \nonumber &|Y_t|^2+\int_t^\tau\,d[Y]_r
\le|Y_\tau|^2+2\int_t^\tau|Y_r|f_r\,dr
+2\int_t^\tau|Y_r|\,d|V|_r\nonumber\\
&\qquad+\frac12\int_t^\tau
\|Z_r\|_{M_r}\,dr-2\sum_{i=1}^\infty\int_t^\tau
Y_{r-}Z^i_r\,dM_r^i.
\end{align}
It is well known that
\begin{align}
\label{eq7.19} & E\int_t^\tau\,d[Y]_r =E\int_t^{\tau}
d[M]_r=E\int_t^\tau\, d\langle M\rangle_r
\end{align}
For $k\in\BN$ write
$\tau_k=\sigma_k\wedge\delta_k$, where $\{\sigma_k\}$ is a
fundamental sequence for the local martingale
$\sum_{i=1}^{\infty}\int_0^\cdot Y_{r-} Z_r^i\,dM_r^i$
and
\[
\delta_k=\inf\{t\ge 0,\int_0^t \|Z_r\|^2_{M_r}\,dr\ge k\}.
\]
Substituting  (\ref{eq7.19}) into (\ref{eq7.13})
and then replacing $\tau$ by $\tau_k$ in (\ref{eq7.13}),
integrating and  letting $k\rightarrow\infty$  we get
\begin{equation}
\label{eq7.14} E|Y_t|^2+\frac12
E\sum_{i=1}^\infty\int_0^T|Z^i|^2_r\,d\langle M^i\rangle_r\le EX,
\end{equation}
where
$X=|\xi|^2+2\int_0^T|Y_r|f_r\,dr+2\int_0^T|Y_r|\,d|V|_r$.
Furthermore, by (\ref{eq7.13}),
\begin{align*}
E\sup_{t\le T}|Y_t|^2&\le EX+2E\sup_{t\le T}
\Big|\sum_{i=1}^\infty\int_t^T
Y_{r-}Z^i_r\,dM^i_r\Big|\\
&\le EX+c_1 2E\sup_{t\le T}|Y_t|
\Big(\sum_{i=1}^\infty\int_t^T
|Z^i_r|^2\,d[M^i]_r\Big)\\
& \le EX+\beta E\sup_{t\le T}|Y_t|^2+\beta^{-1} c_1 2
E\sum_{i=1}^\infty\int_t^T |Z^i_r|^2\,d\langle M^i\rangle_r.
\end{align*}
Taking $\beta>0$ sufficiently small and using (\ref{eq7.14}) we
obtain
\[
E\sup_{t\le T}|Y_t|^p+\frac12
E\sum_{i=1}^\infty\int_0^T|Z^i|^2_r\,d\langle M^i\rangle_r\le
c_2EX.
\]
Combining this with the estimate
\[
EX\le\beta E\sup_{t\le T}|Y_t|^2+\beta^{-1}c_4
E\Big(|\xi|^2+\big(\int_0^Tf_r\,dr\big)^2+\big(\int_0^Td|V|_r\big)^2\Big)
\]
we get the desired result.
\end{dow}
\medskip

\begin{uw}
If $(Y,Z,R)=(Y^1,Z^1,R^1)-(Y^2,Z^2,R^2)$, where $(Y^i,Z^i,R^i)$,
is a solution of RBSDE$(\xi^i,f^i+dV^i,L,U)$, $i=1,2$, then  from
Propositions \ref{stw5.1}, \ref{stw5.2} and condition (c) of the
definition of a solution of RBSDE it follows that for every
$\alpha\in\mathbb{R}$ the pair $(Y^{\alpha},Z^{\alpha})$ satisfies
(\ref{eq.al7})  with $\xi=\xi^1-\xi^2$,
$f(r,y,z)=f^1(r,y+Y^2_r,z+Z^2_r)-f^2(r,Y^2_r,Z^2_r)$, $V=V^1-V^2$.
We will use this fact in the sequel of the paper without further
explanations.
\end{uw}

\begin{stw}
Assume \mbox{\rm{(A4)}}. Then there exists at most one solution
$(Y,Z,R)$ of RBSDE $(\xi,f+dV,L,U)$ such that $Y\in \mathcal{S}^2$.
\end{stw}
\begin{dow}
Follows immediately from Proposition \ref{stw7.2}.
\end{dow}

\begin{tw}
\label{tw5.1} Assume $\mbox{\rm{(A1)--(A6)}}$.
Then there exists a solution $(Y,Z,R)\in \mathcal{S}^2\otimes
\mathcal{M}^2\otimes \mathcal{V}^2$ of
$\mbox{\rm{RBSDE}}(\xi,f+dV,L,U)$.
\end{tw}
\begin{dow}
Let us define
\[
\Phi=(\Phi^1,\Phi^2):\mathcal{S}^2\otimes M^2\mapsto \mathcal{S}^2\otimes M^2
\]
as follows: for every $(X,H)\in \mathcal{S}^2\otimes M^2$,
$(\Phi^1(X,H),\Phi^2(X,H))$ are the first two components of the
solution of RBSDE$(\xi,f_H+dV,L,U)$ with
$f_H(t,x,y)=f(t,x,y,H_t)$. By Theorem \ref{tw3.p} the mapping
$\Phi$ is well defined. Let $(X^i,H^i)\in \mathcal{S}^2\otimes
M^2$, $i=1,2$, and let $(X,H)=(X^1,H^1)-(X^2,H^2)$,
$(Y^i,Z^i)=\Phi(X^i,H^i)$, $i=1,2$ and
$(Y,Z)=(Y^1,Z^1)-(Y^2,Z^2)$. Observe that
\[
Y_t=\int_t^TF(r,Y_r)\,dr+\int_t^T\,dR_r^1-dR_r^2-
\sum_{i=1}^\infty\int_t^TZ^i_r\,dM_r^i,\quad t\in[0,T],
\]
where
\[
F(r,y)=f(r,y+Y^2_r,H^1_r)-f(r,Y^2_r,H^2_r)
\]
and $R^i$ is the finite variation process such that the triple
$(Y^i,Z^i,R^i)$ is a solution of RBSDE$(\xi,f_H+dV,L,U),\, i=1,2$.
By Proposition \ref{stw7.2} and (A4),
\begin{align}
\label{eq5.18} \nonumber &E\sup_{t\le
T}|Y_t|^2+E\Big(\sum_{i=1}^\infty\int_0^T|Z^i_r|^2\,d \langle
M^i\rangle_r\Big)\le CE\Big(\int_0^T|F(r,0)|\,dr\Big)^2\\
&\quad\le\lambda CE\Big(\int_0^T\|H_r\|_{M_r}\,dr\Big)^2\le
\lambda^2CTE\Big(\sum_{i=1}^\infty\int_0^T|H^i_r|^2\,d
\langle M^i\rangle_r\Big).
\end{align}
It follows that $\Phi$ is a contraction on $\mathcal{S}^2\otimes
M^2$ for a sufficiently small $T$, so  using Banach's principle we
can construct unique solutions on small intervals. Therefore
dividing the interval $[0,T]$ into a finite number of small
intervals and using the standard arguments we can construct a
solution $(Y,Z,R)$ of RBSDE$(\xi,f+dV,L,U)$ on the whole interval
$[0,T]$. Of course $(Y,Z)\in \mathcal{S}^2\otimes M^2$.
Consequently, $R\in \mathcal{V}^2_0$ by Theorem \ref{tw3.p}.
\end{dow}
\medskip

To define solutions of equations with generators depending on $z$
one can use other than (\ref{eq5.1}) types of representation
theorems for martingales. One possibility is outlined below.

\begin{uw}
\label{uw5.1} Put $E=\BR^l\setminus\{0\}$ for some $l\ge 1$. Let
$B$ be a $d$-dimensional Wiener process and $N$ be an independent
of $B$ Poisson random measure on $\BR_+\times E$ with the
compensator $\nu(dt,de)=dt\otimes\lambda(de)$ such that
$\int_{E}1\wedge|e|^{2}\,\lambda(de)<\infty$. For $t\in [0,T]$,
$B\in\mathcal{B}(E)$ let us put $\tilde{N}([0,t]\times
B)=N([0,T]\times B)-\nu([0,t]\times B)$.  It is known (see
\cite{TangLi}) that every locally square integrable martingale $M$
has the representation
\begin{equation}
\label{eq5.r} M_t=M_0+\int_0^t Z_r\,dB_r
+\int_0^t\!\int_{E}H_r(e)\,\tilde{N}(dr,de), \quad t\in[0,T]
\end{equation}
for some predictable $\BRD$-valued (resp. $L^2(E,\lambda)$-valued)
process $Z$ (resp.  $H$). It is also known that the filtration
$\FF$ generated by $(B,N)$ is quasi-left continuous. Let
\[
f:\Omega\times[0,T]\times\BR\times \BRD\times L^2(E,\lambda)\rightarrow\BR
\]
be a measurable function such that $f(\cdot,y,z,v)$ is
progressively  measurable for every $(y,z,v)\in\BR\times\BRD\times
L^2(E,\lambda)$. After replacing representation (\ref{eq5.1}) by
(\ref{eq5.r}), we may define a solution of RBSDE$(\xi,f+dV,L,U)$
as a quadruple $(Y,Z,H,R)$ such that the triple $(Y,M,R)$ with $M$
given by (\ref{eq5.r}) is a solution of
RBSDE$(\xi,\hat{f}+dV,L,U)$ with
\[
\hat{f}(t,y)=f(t,y,Z_t,H_t),\quad (t,y)\in [0,T]\times\BR.
\]
If we now replace the norm $\|\cdot\|_{M_t}$ by the norm
$\|\cdot\|_{\sim}$ on $\BRD\times L^2(E,\lambda)$ given by
\[
\|(z,v)\|_{\sim}=|z|+\|v\|_{L^2(E,\lambda)}, \quad
(z,v)\in\BRD\times L^2(E,\lambda)
\]
and then repeat step by step the proofs of Proposition
\ref{stw7.2} and Theorem \ref{tw5.1} (with obvious changes) we
will get the existence and uniqueness results for solutions of
reflected BSDEs in the set-up of the definition given above.
\end{uw}


\begin{thebibliography}{99}

\bibitem{BDHPS}
Briand, Ph., Delyon, B., Hu, Y., Pardoux, E. and  Stoica, L.
(2003). $L^{p}$ solutions of Backward Stochastic Differential
Equations. {\em Stochastic Process. Appl.} {\bf 108} 109--129.


\bibitem{CohenElliott}
Cohen, S.N. and Elliott, R.J. (2012). Existence, Uniqueness and
Comparisons for BSDEs in General Spaces. {\em  Ann. Probab.} {\bf
40} 2264--2297.

\bibitem{DavisVaraiya}
Davis, M.H.A. and Varaiya, P. (1974). The Multiplicity of an
Increasing Family of $\sigma$-Fields. {\em Ann. Probab.} {\bf 2}
958--963.


\bibitem{ElKaroui}
El Karoui, N. (1981). Les aspects probabilistes du contr\^ole
stochastique. {\em Lecture Notes in Math. } {\bf 876} (1981)
73--238.

\bibitem{Hamadene}
Hamad\`ene, S. (2002).  Reflected BSDEs with discontinuous barrier
and applications. {\em Stochastics Stochastics Rep.} {\bf 74}
571--596.

\bibitem{HamadeneHassani}
Hamad\`ene, S. and  Hassani, M. (2006). BSDEs with two reflecting
barriers driven by a Brownian and a Poisson noise and related
Dynkin game. {\em Electron. J. Probab.} {\bf 11} 121--145.

\bibitem{HamadenePopier}
Hamad\`ene, S. and Popier, A. (2012) $L^p$-solutions for reflected
backward stochastic differential equations. {\em Stoch. Dyn.} {\bf
12} 1150016, 35 pp.

\bibitem{HamadeneWang}
Hamad\`ene, S. and Wang, H. (2009). BSDEs with two RCLL reflecting
obstacles driven by Brownian motion and Poisson measure and
related mixed zero-sum game. {\em Stochastic Process.  Appl.} {\bf
119} 2881--2912.

\bibitem{Jacod}
Jacod, J. (1981). Convergence en loi de semimartingales et
variation quadratique, {\em Lecture Notes in Math.} {\bf 850}
547--560.

\bibitem{Kl:EJP}
Klimsiak, T. (2012) Reflected BSDEs with monotone generator. {\em
Electron. J. Probab.} {\bf 17}, no. 107, 1--25.

\bibitem{Kl:BSM}
Klimsiak, T. (2013). BSDEs with monotone generator and two
irregular reflecting barriers. {\em Bull. Sci. Math.} {\bf 137}
268--321.

\bibitem{Kl:JFA}
Klimsiak, T. (2014). Semi-Dirichlet forms, Feynman-Kac functionals 
and the Cauchy problem for semilinear parabolic equations. Available at
arXiv:1401.3643.

\bibitem{KR:JFA}
Klimsiak, T. and Rozkosz, A. (2013). Dirichlet forms and
semilinear elliptic equations with measure data. {\em J. Funct.
Anal.}  {\bf 265} 890--925.


\bibitem{KR:SM}
Klimsiak, T. and Rozkosz, A. (2013). Semilinear elliptic equations
with measure data and quasi-regular Dirichlet forms. Available at
arXiv:1307.0717.

\bibitem{LepeltierMatoussiXu}
Lepeltier, J.P., Matoussi, A. and Xu, M. (2005). Reflected
bakcward stochastic differential equations under monotonicity and
general increasing growth conditions. {\em Adv. in Appl. Probab.}
{\bf 37} 134--159.

\bibitem{Lepeltier}
Lepeltier, J.P. and Xu, M. (2007) Reflected BSDEs with two rcll
barriers. {\em ESAIM Probab. Stat.} {\bf 11} 3--22.

\bibitem{LLQ}
Liang, G., Lyons, T. and Qian, Z. (2011). Backward Stochastic
Dynamics on a Filtered Probability Space. {\em Ann. Probab.} {\bf
39} 1422--1448.

\bibitem{MeminSlominski}
M\`emin, J. and S\l omi\'nski, L. (1991).  Condition UT et
stabilit\`e en loi des solutions d'\`equations diff\`erentialles
stochastiques. In: {\em  S\`eminaire de Probabilit\`es XXV,
Lecture Notes in Math.} {\bf 1485}, Springer, Berlin.

\bibitem{NualartSchoutens}
Nualart, D. and Schoutens, W. (2001). Backward stochastic
differential equations and Feynman-Kac formula for L\`evy
processes, with applications in finance. {\em Bernoulli} {\bf 7}
761--776.

\bibitem{Peng}
Peng, S. (1999). Monotonic Limit Theorem of BSDE and Nonlinear
Decomposition Theorem of Doob-Meyers Type. {\em Probab. Theory
Related Fields} {\bf 113} 473--499.

\bibitem{PengXu}
Peng, S. and Xu, M. (2005). The Smallest g-Supermartingale and
Reflected BSDE with Single and Double L2 obstacles. {\em   Ann.
Inst. H. Poincar\'e Probab. Statist.}  {\bf 41} 605--630.


\bibitem{Protter}
Protter, P. (2004). {\em Stochastic Integration and Differential
Equations}. 2nd ed. Springer, Berlin.

\bibitem{RevuzYor}
Revuz, D. and Yor, M. (1991). {\em Continuous Martingales and
Brownian Motion.} Springer, Berlin.

\bibitem{RS:SPA}
Rozkosz, A. and S\l omi\'nski, L. (2012). $L^p$ solutions of
reflected BSDEs under monotonicity condition. {\em Stochastic
Process. Appl.} {\bf 122} 3875--3900.

\bibitem{Stoyanov}
Stoyanov, J. (1997). Counterexamples in Probability. 2ed., Wiley,
Chichester.

\bibitem{TangLi}
Tang, S. and Li, X. (1994). Necessary Conditions for Optimal
Control of Stochastics Systems with Random Jumps. {\em  SIAM J.
Control Optim.} {\bf 32} 1447--1475.

\end{thebibliography}
\end{document}